\UseRawInputEncoding
\documentclass[12pt]{article}
\usepackage[centertags]{amsmath}
\usepackage{amsfonts}
\usepackage{amssymb}
\usepackage{latexsym}
\usepackage{amsthm}
\usepackage{newlfont}
\usepackage{graphicx}
\usepackage{listings}
\usepackage{booktabs}
\usepackage{abstract}
\usepackage{enumerate}
\usepackage{xcolor}
\RequirePackage{srcltx}
\date{}

\newlength{\defbaselineskip}
\newcommand{\setlinespacing}[1]%
           {\setlength{\baselineskip}{#1 \defbaselineskip}}

\newcommand{\N}{{\mathbb{N}}}

\newcommand{\actaqed}{\hfill $\actabox$}
{\medskip\noindent \textit{Proof of #1. }}%
{\actaqed \medskip}

\def\D{{\mathcal D}}

\def\A{{\mathcal A}}

\def\C{{\mathcal C}}
\def\cC{{\mathcal C}}

\def \Tr{\mathcal T}

\def \cX{\mathcal X}

\def \cM{\mathcal M}
\def\R{{\mathbb R}}
\def\Z{\mathbb Z}

\def \T{\mathbb T}

\def\bbC{\mathbb C}
\def \<{\langle}
\def\>{\rangle}

\def \Og{\Omega}
\def \ep{\epsilon}

\def \va{\varepsilon}

\def \ff{\varphi}
\def\al{\alpha}
\def\bt{\beta}

\def \sp{\operatorname{span}}

\def\bx{\mathbf x}

\def\bk{\mathbf k}

\def\bu{\mathbf u}

\def\ba{\mathbf a}

\def\bN{\mathbf N}

\def\bF{\mathbf F}
\def\bU{\mathbf U}
\def\bA{\mathbf A}

\def\bt{\beta}

\newtheorem{Theorem}{Theorem}[section]
\newtheorem{Lemma}{Lemma}[section]
\newtheorem{Definition}{Definition}[section]
\newtheorem{Proposition}{Proposition}[section]
\newtheorem{Remark}{Remark}[section]

\numberwithin{equation}{section}

\newcommand{\be}{\begin{equation}}
\newcommand{\ee}{\end{equation}}

\def\al{{\alpha}}

\def\Bl{\Bigl}
\def\Br{\Bigr}
\def\f{\frac}

\def\vi{\varphi}

\def\va{\varepsilon}

\def\CD{{\mathcal D}}

\def\CC{{\mathbb C}}

\def\NN{{\mathbb N}}

\def\Og{\Omega}

\def\sub{\substack}

\def\cX{\mathcal{X}}

\def\bu{\mathbf{u}}

\def\spn{\operatorname{span}}
\def\bx{\mathbf{x}}

\def\bx{\mathbf{x}}

\DeclareFontEncoding{FMS}{}{}
\DeclareFontSubstitution{FMS}{futm}{m}{n}
\DeclareFontEncoding{FMX}{}{}
\DeclareFontSubstitution{FMX}{futm}{m}{n}
\DeclareSymbolFont{fouriersymbols}{FMS}{futm}{m}{n}
\DeclareSymbolFont{fourierlargesymbols}{FMX}{futm}{m}{n}
\DeclareMathDelimiter{\VT}{\mathord}{fouriersymbols}{152}{fourierlargesymbols}{147}

\begin{document}

\title{Sparse sampling recovery by greedy algorithms}

\author{    V. Temlyakov 	\footnote{
		This work was supported by the Russian Science Foundation under grant no. 23-71-30001, https://rscf.ru/project/23-71-30001/, and performed at Lomonosov Moscow State University.
  }}

\newcommand{\Addresses}{{
  \bigskip
  \footnotesize


  V.N. Temlyakov, \textsc{ Steklov Mathematical Institute of Russian Academy of Sciences, Moscow, Russia;\\ Lomonosov Moscow State University; \\ Moscow Center of Fundamental and Applied Mathematics; \\ University of South Carolina.
  \\
E-mail:} \texttt{temlyakovv@gmail.com}

}}
\maketitle

\begin{abstract}{In this paper we analyze approximation and recovery properties with respect to systems satisfying universal sampling discretization property and a special incoherence property.   We apply a powerful nonlinear approximation method -- the Weak Chebyshev Greedy Algorithm (WCGA). We establish that the WCGA based on good points for the $L_p$-universal discretization provides good recovery in the $L_p$ norm. For our recovery algorithms we obtain both the Lebesgue-type inequalities for individual functions and the error bounds for special classes of multivariate functions. 

The main point of the paper is that we combine here two deep and powerful techniques -- Lebesgue-type inequalities for the WCGA and theory of the universal sampling dicretization -- in order to obtain new results in sampling recovery.
		}
\end{abstract}


{\it MSC classification 2000:} Primary 65J05; Secondary 42A05, 65D30, 41A63.

\section{Introduction}
\label{I}

This paper is a follow up to the recent papers \cite{DTM1}--\cite{DTM3}. We discuss sampling recovery, when the error is measured in the $L_p$ norm with $1\le p<\infty$. Recently an outstanding progress has been done in the sampling recovery in the $L_2$ norm (see, for instance, \cite{CM}, \cite{KU}, \cite{KU2}, \cite{NSU}, \cite{CoDo}, \cite{KUV}, \cite{TU1}, \cite{LimT}, \cite{VT183}, \cite{JUV}, \cite{DKU}). In this paper we only discuss those known results, which are directly related to our new results. 
It is known that 
in the case of recovery in the $L_2$ norm different versions of the classical least squares algorithms are very useful. In the paper \cite{DTM1} it has been discovered  that results on universal sampling discretization of the square norm are useful in sparse sampling recovery with error being  measured in the square norm. 
It was established in \cite{DTM2} that a simple greedy type algorithm -- Weak Orthogonal Matching Pursuit -- based on good points for universal 
discretization provides effective recovery in the square norm.
In the paper \cite{DTM3} we extended those results by replacing the square norm with other integral norms. In this case we need to conduct our analysis in a Banach space rather than in
 a Hilbert space, which makes the techniques more involved. In particular, we established in \cite{DTM3} that in the case of special systems a greedy type algorithm -- Weak Chebyshev Greedy Algorithm (WCGA) -- based on good points for the $L_p$-universal discretization provides good recovery in the $L_p$ norm.
We continue to study the power of the WCGA in the nonlinear (sparse) sampling recovery setting.  In this paper we analyze approximation and recovery properties with respect to systems satisfying universal sampling discretization property and a special incoherence property.  We establish that in the case of systems satisfying certain conditions (see Definitions \ref{ID1} and \ref{ID2}) the WCGA based on good points for the $L_p$-universal discretization provides good recovery in the $L_p$ norm. For our recovery algorithms we obtain both the Lebesgue-type inequalities for individual functions and the error bounds for special classes of multivariate functions. 
The reader can find a survey of results on discretization in \cite{DPTT} and \cite{KKLT}. 

 We begin with a brief 
description of  some  necessary concepts on 
sparse approximation.   Let $X$ be a Banach space with norm $\|\cdot\|:=\|\cdot\|_X$, and let $\D=\{g_i\}_{i=1}^\infty $ be a given (countable)  system of elements in $X$. Given a finite subset $J\subset \NN$, we define $V_J(\D):=\spn\{g_j:\  \ j\in J\}$. 
For a positive  integer $ v$, we denote by $\mathcal{X}_v(\D)$ the collection of all linear spaces $V_J(\D)$  with   $|J|=v$, and 
denote by $\Sigma_v(\D)$  the set of all $v$-term approximants with respect to $\D$; that is, 
$
\Sigma_v(\D):= \bigcup_{V\in\cX_v(\D)} V.
$
Given $f\in X$,  we define
$$
\sigma_v(f,\D)_X := \inf_{g\in\Sigma_v(\D)}\|f-g\|_X,\  \ v=1,2,\cdots.
$$ 
Moreover,   for a function class $\bF\subset X$, we define 
$$
 \sigma_v(\bF,\D)_X := \sup_{f\in\bF} \sigma_v(f,\D)_X,\quad  \sigma_0(\bF,\D)_X := \sup_{f\in\bF} \|f\|_X.
 $$
 
 In this paper we consider the case $X=L_p(\Omega,\mu)$, $1\le p<\infty$. More precisely, let $\Omega$ be a compact subset of $\R^d$ with the probability measure $\mu$ on it. By the $L_p$ norm, $1\le p< \infty$, of the complex valued function defined on $\Omega$,  we understand
$$
\|f\|_p:=\|f\|_{L_p(\Omega,\mu)} := \left(\int_\Omega |f|^pd\mu\right)^{1/p}.
$$
In the case $X=L_p(\Omega,\mu)$ we sometimes write for brevity $\sigma_v(\cdot,\cdot)_p$ instead of 
$\sigma_v(\cdot,\cdot)_{L_p(\Omega,\mu)}$.

Mostly, we study systems, which have two special properties.  The first property given in the Definition \ref{ID1} concerns the universal sampling discretization. 
 
  \begin{Definition}\label{ID1} Let $1\le p<\infty$. We say that a set $\xi:= \{\xi^j\}_{j=1}^m \subset \Omega $ provides {\it  $L_p$-universal sampling discretization}   for the collection $\cX:= \{X(n)\}_{n=1}^k$ of finite-dimensional  linear subspaces $X(n)$ if we have
 \be\label{ud}
\frac{1}{2}\|f\|_p^p \le \frac{1}{m} \sum_{j=1}^m |f(\xi^j)|^p\le \frac{3}{2}\|f\|_p^p\quad \text{for any}\quad f\in \bigcup_{n=1}^k X(n) .
\ee
We denote by $m(\cX,p)$ the minimal $m$ such that there exists a set $\xi$ of $m$ points, which
provides  the $L_p$-universal sampling discretization (\ref{ud}) for the collection $\cX$. 

We will use a brief form $L_p$-usd for the $L_p$-universal sampling discretization (\ref{ud}).
\end{Definition}

The second property given in the Definition \ref{ID2} concerns incoherence property of the system, which is 
known to be useful in approximation by the WCGA (see, for instance, \cite{VTbookMA}, Section 8.7). 

\begin{Definition}\label{ID2}  Let $X$ be a Banach space with a norm $\|\cdot\|$. We say that a system $\D=\{g_i\}_{i=1}^\infty\subset X$ has   ($v,S$)-incoherence property with parameters   $V>0$ and $r>0$ in $X$  if for any $A\subset B$ with  $|A|\le v$ and  $|B|\le S$,  and  for any $\{c_i\}_{i\in B}\subset \CC$, we have
\be\label{incoh}
\sum_{i\in A} |c_i| \le V|A|^r\Bl\|\sum_{i\in B} c_ig_i\Br\|.
\ee
We will use a brief form ($v,S$)-ipw($V,r$)   for the ($v,S$)-incoherence property with parameters   $V>0$ and $r>0$ in $X$.
\end{Definition}

We gave Definition \ref{ID2} for a countable system $\D$. Similar definition can be given for any system $\D$ as well. 

In this paper we analyze approximation and recovery properties with respect to systems satisfying 
Definitions \ref{ID1} and \ref{ID2}. We concentrate on application of a powerful nonlinear approximation method -- the Weak Chebyshev Greedy Algorithm (WCGA) (see Section \ref{CG}). In Section \ref{CG}, using known results on the Lebesgue-type inequalities for the WCGA, we prove similar inequalities for
a discretized version of the WCGA. Namely, we use the WCGA in the space $L_p(\Omega_m,\mu_m)$
 instead of the space $L_p(\Omega,\mu)$, where $\Omega_m=\{\xi^\nu\}_{\nu=1}^m$ is from Definition \ref{ID1} and  $\mu_m(\xi^\nu) =1/m$, $\nu=1,\dots,m$. Let $\D_N(\Omega_m)$ be the restriction 
of $\D_N$ onto $\Omega_m$.
 For instance, in Section \ref{CG} we prove under assumption that $\D_N$ has the $(v,S)$-ipw$(V,1/2)$ property
 (see Theorem \ref{CGT2}) that for any given $f_0\in \cC(\Omega)$,  the WCGA with weakness parameter $t$ applied to $f_0$ with respect to the system  $\D_N(\Omega_m)$ in the space $L_p(\Omega_m,\mu_m)$, $2\le p<\infty$, provides the following bound on the error for the residual 
$$
\|f_{c (2V)^{2}(\ln (2Vv)) v}\|_{L_p(\Omega,\mu)} \le C\sigma_v(f_0,\D_N)_\infty,
$$
where $c$ is a constant, which may depend on $p$ and on the weakness parameter $t$ of the WCGA, $C\ge 1$ is an absolute constant, and $V$ is a characteristic of $\D_N$. 

In Section \ref{Ex} we discuss specific systems $\D_N$, which satisfy Definitions \ref{ID1} and \ref{ID2}. 
For instance, a typical example of such system is a uniformly bounded orthonormal system. 

The idea of using the WCGA in the space $L_p(\Omega_m,\mu_m)$ defined on the discrete domain $\Omega_m$ 
 instead of the space $L_p(\Omega,\mu)$ is motivated by applications to the sampling recovery problem.
 Indeed, the WCGA applied in the space $L_p(\Omega_m,\mu_m)$ only uses function values at points 
 from $\Omega_m$. In Section \ref{R} we work this idea out on the example of classes of multivariate functions. We introduce and study function classes defined by structural assumptions rather than by 
 smoothness assumptions. 
 
 The main point of the paper is that we combine here two deep and powerful techniques --  Lebesgue-type inequalities for the WCGA and  theory of the universal sampling dicretization -- in order to obtain
 new results in sampling recovery. We prove here two kinds of results -- the Lebesgue-type inequalities for 
 recovery by the WCGA (Sections \ref{CG} and \ref{Ex}) and bounds on optimal errors of recovery for some function classes (Section \ref{R}). 
 As we already pointed out the universal sampling discretization plays an important role in our study.
 In Section \ref{G} we show that the problem of universal sampling discretization is equivalent to another 
 problem, which can be formulated as a generalization of the well known problem on the Restricted Isometry Property (RIP). 
 
 We use $C$, $C'$ and $c$, $c'$ to denote various positive constants. Their arguments indicate the parameters, which they may depend on. Normally, these constants do not depend on a function $f$ and running parameters $m$, $v$, $u$. We use the following symbols for brevity. For two nonnegative sequences $a=\{a_n\}_{n=1}^\infty$ and $b=\{b_n\}_{n=1}^\infty$ the relation $a_n\ll b_n$ means that there is  a number $C(a,b)$ such that for all $n$ we have $a_n\le C(a,b)b_n$. Relation $a_n\gg b_n$ means that 
 $b_n\ll a_n$ and $a_n\asymp b_n$ means that $a_n\ll b_n$ and $a_n \gg b_n$. 
 For a real number $x$ denote $[x]$ the integer part of $x$, $\lceil x\rceil$ -- the smallest integer, which is 
 greater than or equal to $x$. 
 
 {\bf Novelty.} In Section \ref{CG} we prove some general results on the sampling recovery (see Theorems \ref{CGT2} and \ref{CGT3}). These results are proved under the condition that the system $\D_N$ has 
 two properties described in Definitions \ref{ID1} and \ref{ID2}. In previous papers recovery results of that type were proved for rather specific systems. In Section \ref{Ex} we demonstrate how those known results can be easily derived from the general Theorems \ref{CGT2} and \ref{CGT3}. 
 
 Important new results are obtained in Section \ref{R}. We apply there our sampling recovery results from Section \ref{Ex} to optimal recovery on specific classes of multivariate functions. 
 For a function class $\bF\subset \cC(\Omega)$,  we  define (see \cite{TWW})
$$
\varrho_m^o(\bF,L_p) := \inf_{\xi } \inf_{\cM} \sup_{f\in \bF}\|f-\cM(f(\xi^1),\dots,f(\xi^m))\|_p,
$$
where $\cM$ ranges over all  mappings $\cM : \bbC^m \to   L_p(\Omega,\mu)$  and
$\xi$ ranges over all subsets $\{\xi^1,\cdots,\xi^m\}$ of $m$ points in $\Og$. 
Here, we use the index {\it o} to mean optimality.  

In Section \ref{R} we study the behavior of $\varrho_m^o(\bF,L_p)$ for a new kind of classes -- 
$\bA^r_\bt(\Psi)$. We obtain in Section \ref{R} the order of $\varrho_m^o(\bA^r_\bt(\Psi),L_p)$ (up to a logarithmic in $m$ factor) in the case $1\le p\le 2$ for uniformly bounded orthonormal systems $\Psi$.
We give a definition of classes $\bA^r_\bt(\Psi)$ right now.  For a given $1\le p\le \infty$,
let $\Psi =\{\psi_{\bk}\}_{\bk \in \Z^d}$, $\psi_\bk \in \cC(\Omega)$, $\|\psi_\bk\|_p \le B$, $\bk\in\Z^d$,  be a system in the space $L_p(\Omega,\mu)$. We consider functions representable in the form of an  absolutely convergent series
\be\label{Irepr}
f = \sum_{\bk\in\Z^d} a_\bk(f)\psi_\bk,\quad \sum_{\bk\in\Z^d} |a_\bk(f)|<\infty.
\ee
For $\bt \in (0,1]$ and $r>0$ consider the following class $\bA^r_\bt(\Psi)$ of functions $f$, which have representations (\ref{Irepr}) satisfying the following conditions
\be\label{IAr}
  \left(\sum_{[2^{j-1}]\le \|\bk\|_\infty <2^j} |a_\bk(f)|^\bt\right)^{1/\bt} \le 2^{-rj},\quad j=0,1,\dots  .
\ee

In a special case, when $\Psi$ is the trigonometric system $\Tr^d := \{e^{i(\bk,\bx)}\}_{\bk\in \Z^d}$ we prove in Section \ref{R} the following lower bound (see Theorem \ref{RT7})
\be\label{Arlb}
\varrho_m^o(\bA^r_\bt(\Tr^d),L_1) \gg m^{1/2-1/\bt-r/d}.
\ee 
 We complement this lower bound by the following upper bound (see Theorem \ref{RT3}). 
 Assume that $\Psi$ is a uniformly bounded $\|\psi_\bk\|_\infty \le B$, $\bk\in\Z^d$, orthonormal system.
Let $1\le p\le 2$, $\bt \in (0,1]$, and $r>0$, then
$$
 \varrho_{m}^{o}(\bA^r_\bt(\Psi),L_p(\Omega,\mu))\le  \varrho_{m}^{o}(\bA^r_\bt(\Psi),L_2(\Omega,\mu))
$$
\be\label{Arub}
  \ll  (m(\log(m+1))^{-5})^{1/2 -1/\bt-r/d} .  
\ee
Note that $(\log(m+1))^{-5}$ in (\ref{Arub}) can be replaced by $(\log(m+1))^{-4}$ (see Theorem \ref{DT4} and Remark \ref{DR1}). 
Bounds (\ref{Arlb}) and (\ref{Arub}) show that for uniformly bounded orthonormal systems $\Psi$, for instance for the trigonometric system $\Tr^d$, the gap between the upper and the lower bounds is 
in terms of an extra logarithmic in $m$ factor. Also, bound (\ref{Arub}) shows that for classes $\bA^r_\bt(\Psi)$ nonlinear sampling 
recovery is much better than the linear sampling recovery (see Theorem \ref{RT5}). 
Application of our results from Section \ref{Ex} in proving upper bounds for optimal recovery requires 
upper bounds on the best $v$-term approximations of the class of interest. We obtain such bounds in 
Section \ref{R} (see Theorem \ref{RT1}). These bounds might be of independent of recovery interest.  

It is well known that classes of functions with restrictions imposed on the coefficients of their expansions 
with respect to a given system (dictionary) are very natural for nonlinear sparse approximation (see, for instance, \cite{VTbook} and \cite{VTbookMA}).  Sparse approximation and sampling recovery of classes with restrictions on coefficients of functions expansions were studied in \cite{VT150} and \cite{JUV}.

 For further discussions see subsection "Comments" in Section \ref{Ex} and Section \ref{D}.

\section{Some general results on recovery by WCGA}
\label{CG}

We give the  definition of  the Weak Chebyshev Greedy Algorithm (WCGA) in a Banach space,  which  was introduced in \cite{T1}  as a generalization  of the Weak Orthogonal Matching Pursuit (WOMP).
To be more precise, 
let $X^\ast$ denote the  dual of the Banach space $X$. 
For a nonzero element $g\in X$,  we denote by  $F_g $  a norming (peak) functional for $g$,  that is,  an element  $F_g\in X^\ast$ satisfying 
$$
\|F_g\|_{X^*} =1,\qquad F_g(g) =\|g\|_X.
$$
The existence of such a functional is guaranteed by the Hahn-Banach theorem.

Now we can define the WCGA as follows.\\

{\bf Weak Chebyshev Greedy Algorithm (WCGA).}  Let
$\tau := \{t_k\}_{k=1}^\infty$ be a given weakness sequence of  positive numbers $\leq 1$. Let $\D=\{g\}\subset X$ be a system of nonzero elements in $X$ such that $\|g\|\le 1$ for $g\in\D$. 
Given  $f_0\in X$, we define  the elements  $f_m\in X$ and $\phi_m\in \CD$  for $m=1,2,\cdots$ inductively   as follows:

\begin{enumerate}[\rm (1)]

	\item  $\phi_m  \in \D$ is any element satisfying
	$$
	|F_{f_{m-1}}(\phi_m)| \ge t_m\sup_{g\in\D}  | F_{f_{m-1}}(g )|.
	$$
	
	\item  Define
	$$
	\Phi(m) := \sp \{\phi_1,\cdots, \phi_m\},
	$$
	and let $G_m  := G_m(f_0, \CD)_X$  be the best approximant to $f_0$ from  the space $\Phi(m)$; that is, 
	$$
	G_m :=\underset {G\in \Phi(m)} {\operatorname{argmin}}\|f_0-G\|_X.
	$$
	
	\item  Define
	$$
	f_m := f_0-G_m.
	$$
	
\end{enumerate}

\begin{Remark}\label{CGR1} We defined the WCGA for a system satisfying an extra condition $\|g\|\le 1$ for $g\in\D$. Clearly, realizations of the WCGA for a new system $\D^B := \{Bg, \, g\in \D\}$, where $B$ is a positive number, coincide with those for the system $\D$. This means that the restriction $\|g\|\le 1$ can be replaced by the restriction $\|g\|\le B$ with some positive number $B$.
\end{Remark}

In this paper we  shall only consider the WCGA in  the case when $t_k=t\in (0, 1]$ for $k=1,2,\dots$. 
We also point out that in the case when $X$ is a Hilbert space,  the WCGA coincides with the well known  WOMP, which  is very popular in signal processing, and in particular, in compressed sensing. In approximation theory the WOMP is also called the Weak Orthogonal Greedy Algorithm (WOGA).

Recall that  the modulus of smoothness of a Banach space  $X$ is defined as 
\begin{equation}\label{CG1}
\eta(w):=\eta(X,w):=\sup_{\sub{x,y\in X\\
		\|x\|= \|y\|=1}}\Bigg[\f {\|x+wy\|+\|x-wy\|}2 -1\Bigg],\  \ w>0,
\end{equation}
and that $X$ is called uniformly smooth  if  $\eta(w)/w\to 0$ when $w\to 0+$.
It is well known that the $L_p$ space with $1< p<\infty$ is a uniformly smooth Banach space with 
\be\label{CG2}
\eta(L_p,w)\le \begin{cases}(p-1)w^2/2, & 2\le p <\infty,\\   w^p/p,& 1\le p\le 2.
\end{cases}
\ee

The following Theorem \ref{CGT1} was proved in \cite{VT144} for real Banach spaces (see also \cite{VTbookMA}, Section 8.7, Theorem 8.7.17, p.431) and in \cite{DGHKT} for complex Banach spaces. Note that this theorem was proved there under condition that $\D$ is a dictionary but its proof works for a system as well. 
 \begin{Theorem}[{\cite[Theorem 2.7]{VT144}, \cite[Theorem 8.7.17]{VTbookMA}}]\label{CGT1} Let $X$ be a Banach space satisfying that  $\eta(X, w)\le \gamma w^q$, $w>0$ for some parameter $1<q\le 2$. Suppose that $\D\subset X$  is a system in $X$ with the  ($v,S$)-incoherence property for some integers $1\leq v\leq S$ and  parameters   $V>0$ and $r>0$.   Then the WCGA with weakness parameter $t$ applied to $f_0$ with respect to the system $\CD$ provides
$$
\|f_{u}\| \le C\sigma_v(f_0,\D)_X,\quad u:=\lceil C(t,\gamma,q)V^{q'}\ln (Vv) v^{rq'}\rceil, 
$$
for any positive integer  $v$ satisfying $v+u \le S,$ where 
$$
q':=\f q{q-1},\   \ C(t,\gamma,q) = C(q)\gamma^{\frac{1}{q-1}}  t^{-q'},
$$ 
and $C>1$ is an absolute constant. 
\end{Theorem}

We will be applying Theorem \ref{CGT1} to the discretized version of the given system $\D_N = \{g_i\}_{i=1}^N$. We now introduce the necessary notations. Let $X_N$ be an $N$-dimensional subspace of the space of continuous functions $\C(\Omega)$. For a fixed $m\in\NN$ and a set of points  $\xi:=\{\xi^\nu\}_{\nu=1}^m\subset \Omega$ we associate with a function $f\in \C(\Omega)$ a vector (sample vector)
$$
S(f,\xi) := (f(\xi^1),\dots,f(\xi^m)) \in \bbC^m.
$$
Denote
$$
\|S(f,\xi)\|_p := \left(\frac{1}{m}\sum_{\nu=1}^m |f(\xi^\nu)|^p\right)^{1/p},\quad 1\le p<\infty,
$$
and 
$$
\|S(f,\xi)\|_\infty := \max_{\nu}|f(\xi^\nu)|.
$$
Typically, instead of the space $L_p(\Omega,\mu)$ we consider the space $L_p(\Omega_m,\mu_m)$
where $\Omega_m=\{\xi^\nu\}_{\nu=1}^m$ is from Definition \ref{ID1} and  $\mu_m(\xi^\nu) =1/m$, $\nu=1,\dots,m$. Let $\D_N(\Omega_m)$ be the restriction 
of $\D_N$ onto $\Omega_m$. Here and elsewhere in the paper,  we often use the notation $\Omega_m$ to denote the set 
$\xi$ in order to emphasize that the set $\xi$ plays the role of a new domain $\Omega_m$ consisting of $m$ points instead of 
the original domain $\Omega$. 

\begin{Proposition}\label{CGP1} Let $1\le p<\infty$. Assume that the system $\D_N = \{g_i\}_{i=1}^N$ has the ($v,S$)-ipw($V,r$) and the $L_p$-usd for the collection $\cX_u(\D_N)$, $v\le u\le S\le N$. Then the system $\D_N(\Omega_m)$, which is the restriction of the $\D_N$ onto $\Omega_m$, has ($v,u$)-ipw($V2^{1/p},r$) and 
\be\label{CG3}
\|g_i\|_{L_p(\Omega_m,\mu_m)} \le (3/2)^{1/p} \|g_i\|_p,\quad i=1,\dots,N.
\ee
\end{Proposition}
\begin{proof} For any $A$ and $B$ such that $A\subset B$, $|A|\le v$, $|B|\le u \le S$ 
the ($v,S$)-ipw($V,r$) property implies
$$
\sum_{i\in A}|c_i| \le V|A|^{r} \left\|\sum_{i\in B} c_i g_i\right\|.
$$
Set $f:=\sum_{i\in B} c_i g_i$ and 
by the $L_p$-usd property for the collection $\cX_u(\D_N)$  continue
$$
\le V|A|^{r}2^{1/p} \left(\frac{1}{m} \sum_{\nu=1}^m |f(\xi^\nu)|^p\right)^{1/p} = V|A|^{r}2^{1/p}\|S(f,\xi)\|_p 
$$
$$
= V2^{1/p}|A|^{r}\|f\|_{L_p(\Omega_m,\mu_m)}.
$$
Thus, $\D_N(\Omega_m)$ has ($v,u$)-ipw($V2^{1/p},r$).

Also, for any $g_i\in \D_N$ we have by the $L_p$-usd
$$
\|g_i\|^p_{L_p(\Omega_m,\mu_m)} = \frac{1}{m} \sum_{\nu=1}^m |g_i(\xi^\nu)|^p \le \frac{3}{2} \|g_i\|_p^p,
$$
which implies the required bound (\ref{CG3}).

\end{proof} 

We now prove a conditional result on the sampling recovery by the WCGA. We call it {\it conditional} because Theorem \ref{CGT2} below is proved under two conditions on the system $\D_N$, which are non-trivial and non-standard conditions. 

\begin{Theorem}\label{CGT2} Let $1<p<\infty$ and $p^* := \min(p,2)$, $q^* := p^*/(p^*-1)$. Assume that the system $\D_N = \{g_i\}_{i=1}^N$ has the ($v,S$)-ipw($V,r$) and the $L_p$-usd for the collection $\cX_u(\D_N)$, $v\le u\le S\le N$.
 
Assume that $\xi=\{\xi^1,\cdots, \xi^m\}\subset \Og $  is a set  of $m$ points
in $\Og$ that provides the   $L_p$-usd  for the collection $\cX_u(\D_N)$.

Then there exists a constant  $c=C(t,p)\ge 1$ depending only on $t$ and $p$ such that for any positive integer $v$  with  $v+v'\leq u$, $v':=\lceil c(2V)^{q^*}(\ln (2Vv)) v^{rq^*}\rceil$,  and for any given $f_0\in \cC(\Omega)$,  the WCGA with weakness parameter $t$ applied to $f_0$ with respect to the system  $\D_N(\Omega_m)$ in the space $L_p(\Omega_m,\mu_m)$ provides
\be\label{mp}
\|f_{v'}\|_{L_p(\Omega_m,\mu_m)} \le C\sigma_v(f_0,\D_N(\Omega_m))_{L_p(\Omega_m,\mu_m)}, 
\ee
and
\be\label{mp2}
\|f_{v'}\|_{L_p(\Omega,\mu)} \le C\sigma_v(f_0,\D_N)_\infty,
\ee
where $C\ge 1$ is an absolute constant.
 \end{Theorem}
 \begin{proof} Theorem \ref{CGT2} is a corollary of Theorem \ref{CGT1} and Proposition \ref{CGP1}. 
 Consider separately two cases (I) $2\le p<\infty$ and (II) $1<p\le 2$. 
 
 {\bf Case (I) $2\le p<\infty$.} By (\ref{CG2}) we have $\eta(L_p,w)\le (p-1)w^2/2,\  \ w>0.$ In our case 
 $p^*=2$ and $q^*=2$. We set $c=C(t,p) := C(t,(p-1)/2,2)$, where $C(t,\gamma,q)$ is from Theorem \ref{CGT1}. Proposition \ref{CGP1} guarantees that we can apply Theorem \ref{CGT1} to the system 
 $\D_N(\Omega_m)$ in the space $L_p(\Omega_m,\mu_m)$. This gives us inequality (\ref{mp}). 
 
 {\bf Case (II) $1<p\le 2$.} By (\ref{CG2}) in this case we have $\eta(L_p,w)\le w^p/p,\  \ w>0.$ Also, $p^*=p$ and $q^*=p/(p-1)=q'$ defined in Theorem \ref{CGT1}. We set $c=C(t,p) := C(t,1/p,p)$, where $C(t,\gamma,q)$ is from Theorem \ref{CGT1}. Proposition \ref{CGP1} guarantees that we can apply Theorem \ref{CGT1} to the system 
 $\D_N(\Omega_m)$ in the space $L_p(\Omega_m,\mu_m)$. This gives us inequality (\ref{mp}). 
 
We now derive (\ref{mp2}) from (\ref{mp}).   Clearly, 
$$
\sigma_v(f_0,\D_N(\Omega_m))_{L_p(\Omega_m,\mu_m)} \le \sigma_v(f_0,\D_N )_\infty.
$$
Let $f\in \Sigma_v(\D_N)$ be such that  $\|f_0-f\|_\infty = \sigma_v(f_0,\D_N)_\infty$. Let us set $v':= \lceil c (2V)^{q^*}(\ln (2Vv)) v^{rq^*}\rceil$ for brevity. Then (\ref{mp}) implies 
$$
\|f - G_{v'}(f_0,\D_N(\Omega_m))\|_{L_p(\Omega_m,\mu_m)} \le \|f-f_0\|_{L_p(\Omega_m,\mu_m)} +\|f_{v'}\|_{L_p(\Omega_m,\mu_m)} 
$$
$$
\le (1+C)\sigma_v(f_0,\D_N)_\infty.
$$
Using that $f - G_{v'}(f_0,\D_N(\Omega_m)) \in \Sigma_u(\D_N)$, by discretization (\ref{ud}) we 
conclude that
\be\label{ub3}
\|f - G_{v'}(f_0,\D_N(\Omega_m))\|_{L_p(\Omega,\mu)} \le 2^{1/p}(1+C)\sigma_v(f_0,\D_N)_\infty.
\ee
Finally,
$$
\|f_{v'}\|_{L_p(\Omega,\mu)} \le \|f-f_0\|_{L_p(\Omega,\mu)} + \|f - G_{v'}(f_0,\D_N(\Omega_m))\|_{L_p(\Omega,\mu)}.
$$
This and (\ref{ub3}) prove (\ref{mp2}).

\end{proof}

\begin{Theorem}\label{CGT3}
	Under the conditions of Theorem \ref{CGT2}, we have 
	\be\label{mp3}
	\|f_{v'} \|_{L_p(\Omega,\mu)} \le C' \sigma_v(f_0,\CD_N)_{L_p(\Og, \mu_\xi)},
	\ee
	where   $v'$ is    from Theorem \ref{CGT2}, $C'$ is a positive absolute constant, and
	$$
	\mu_\xi := \f {\mu+\mu_m}2=\frac{1}{2} \mu + \frac{1}{2m} \sum_{j=1}^m \delta_{\xi^j}.
	$$
\end{Theorem}
\begin{proof}
	For convenience, we use the notation $\|\cdot\|_{L_p(\nu)}$  to denote the norm of $L_p$ defined with respect to a measure $\nu$ on $\Og$. 
	Let $g\in \Sigma_v(\CD_N)$ be such that  $\|f_0-g\|_{L_p(\mu_\xi)} = \sigma_v(f_0,\CD_N)_{L_p(\mu_\xi)}$.
	Denote as above $v':=\lceil c(2V)^{q^*} (\ln (2Vv)) v^{rq^*}\rceil$. 
	Then 
	\begin{align*}
	\|f_{v'} \|_{L_p(\mu)}&\leq
	2^{1/p} \|f_0 - G_{v'}(f_0,\CD_N(\Omega_m))\|_{L_p(\mu_\xi)}\\
	&\leq 2^{1/p}\|f_0-g\|_{L_p(\mu_\xi)}+2^{1/p} \|g- G_{v'}(f_0,\CD_N(\Omega_m))\|_{L_p(\mu_\xi)}\\
	&\leq 2^{1/p} \sigma_v(f_0,\CD_N)_{L_p(\mu_\xi)}+ 2^{1/p}\|g- G_{v'}(f_0,\CD_N(\Omega_m))\|_{L_p(\mu_\xi)}.
	\end{align*}
	Since
	\[ g- G_{v'}(f_0,\CD_N(\Omega_m))\in \Sigma_{v+v'} (\CD_N) \subset \Sigma_u(\CD_N),\]
	it follows by the $L_p$-usd assumption that 
	\begin{align*} 
	\|g-& G_{v'}(f_0,\CD_N(\Omega_m))\|_{L_p(\mu_\xi)}\leq C_1 \|g- G_{v'}(f_0,\CD_N(\Omega_m))\|_{L_p(\mu_m)}\\
	&	\leq C_1\|f_0-g\|_{L_p(\mu_m)}+C_1 \|f_0- G_{v'}(f_0,\CD_N(\Omega_m))\|_{L_p(\mu_m)}\\
	&\leq 2^{1/p}C_1 \|f_0-g\|_{L_p(\mu_\xi)}+ C_1 \|f_{v'} \|_{L_p(\mu_m)},\end{align*}
	which, by Theorem \ref{CGT2}, is estimated by 
	\begin{align*}
	&\leq C_2  \sigma_v(f_0,\CD_N)_{L_p(\mu_\xi)}+ C_3\sigma_v(f_0,\CD_N(\Omega_m))_{L_p(\mu_m)}\leq C' \sigma_v(f_0,\CD_N)_{L_p(\mu_\xi)}.
	\end{align*}

\end{proof}

\section{Applications of Theorem \ref{CGT2}}
\label{Ex}

Most of results of this section were obtained in \cite{DTM3}, where they were proved directly for the corresponding systems $\D_N$.
In this section we demonstrate how Theorems \ref{CGT2} and \ref{CGT3} can be used for specific 
systems $\D_N$. In order to apply Theorems \ref{CGT2} and \ref{CGT3} to a given system $\D_N$ 
we need to check that this system has both the ($v,S$)-ipw($V,r$) and the $L_p$-usd for the collection $\cX_u(\D_N)$, $v\le u\le S\le N$. We begin with known results on the $L_p$-usd for the collection $\cX_u(\D_N)$. We formulate a known result from \cite{DT}, which established existence of good points for universal discretization.  We now proceed to a special case when   $ \D_N $ is a uniformly bounded Riesz system. Namely, we assume that $\D_N:=\{\ff_j\}_{j=1}^N$ is a system of $N$  uniformly bounded functions on $\Og \subset \R^d$ such that
\be \label{ub}
\sup_{\bx\in\Og} |\vi_j(\bx)|\leq 1,\   \ 1\leq j\leq N.
\ee
Also, we assume that
for any $(a_1,\cdots, a_N) \in\bbC^N,$
\begin{equation}\label{Riesz}
R_1 \left( \sum_{j=1}^N |a_j|^2\right)^{1/2} \le \left\|\sum_{j=1}^N a_j\ff_j\right\|_2 \le R_2 \left( \sum_{j=1}^N |a_j|^2\right)^{1/2},
\end{equation}
where $0< R_1 \le R_2 <\infty$.  

	\begin{Theorem}[{\cite{DT}}]\label{ExT1}   Assume that $\D_N$ is a uniformly bounded Riesz system  satisfying (\ref{ub}) and \eqref{Riesz} for some constants $0<R_1\leq R_2<\infty$.
		Let $2<p<\infty$ and let  $1\leq u\leq N$ be an integer. 		
		 Then for a large enough constant $C=C(p,R_1,R_2)$ and any $\va\in (0, 1)$,   there exist 
		$m$ points  $\xi^1,\cdots, \xi^m\in  \Og$  with 
\be\label{m}
		m\leq C\va^{-7}       u^{p/2} (\log N)^2,
\ee
			such that for any $f\in  \Sigma_u(\D_N)$, 
		\[ (1-\va) \|f\|_p^p \leq \frac   1m \sum_{j=1}^m |f(\xi^j)|^p\leq (1+\va) \|f\|_p^p. \]	
	\end{Theorem}

Theorem \ref{ExT1} covers the case $2<p<\infty$. We now discuss the case $1\le p\le 2$. In this case   we formulate the following recent result from \cite{DTM2}. In that theorem we use the Bessel system assumption instead of the Riesz system assumption (\ref{Riesz}): There exists a constant $K>0$ such that   for any $(a_1,\cdots, a_N) \in\bbC^N,$
\begin{equation}\label{Bessel}
  \sum_{j=1}^N |a_j|^2 \le K  \left\|\sum_{j=1}^N a_j\ff_j\right\|^2_2 .
\end{equation}

   \begin{Theorem}[{\cite{DTM2}}]\label{ExT2} Let $1\le p\le 2$. Assume that $ \D_N=\{\ff_j\}_{j=1}^N\subset \cC(\Og)$ is a  system  satisfying  the conditions  \eqref{ub} and   \eqref{Bessel} for some constant $K\ge 1$. Let $\xi^1,\cdots, \xi^m$ be independent 
 	random points on $\Og$  that are  identically distributed  according to  $\mu$. 
 	 Then there exist constants  $C=C(p)>1$ and $c=c(p)>0$ such that 
 	  given any   integers  $1\leq u\leq N$ and 
 	 $$
 	 m \ge  C Ku \log N\cdot (\log(2Ku ))^2\cdot (\log (2Ku )+\log\log N),
 	 $$
 	 the inequalities 
 	 \begin{equation}\label{Ex2}
 	 \frac{1}{2}\|f\|_p^p \le \frac{1}{m}\sum_{j=1}^m |f(\xi^j)|^p \le \frac{3}{2}\|f\|_p^p,\   \   \ \forall f\in  \Sigma_u(\D_N)
 	 \end{equation}
 hold with probability $\ge 1-2 \exp\Bl( -\f {cm}{Ku\log^2 (2Ku)}\Br)$.
\end{Theorem}
  
Note that the inequalities (\ref{Riesz}) imply the inequality (\ref{Bessel}) with $K=R_1^{-2}$. 

\begin{Remark}\label{ExR1} Condition (\ref{ub}) with constant $1$ is made for convenience. If instead of (\ref{ub}) we assume that $\|\vi_j\|_\infty \le B$, $j=1,\dots,N$, then 
Theorems \ref{ExT1} and \ref{ExT2} hold with constants allowed to depend on $B$. 
\end{Remark}

Theorems \ref{ExT1} and \ref{ExT2} provide the $L_p$-usd for the collection $\cX_u(\D_N)$ under assumption that the system $\D_N$ is a uniformly bounded Riesz system for $1\le p<\infty$. We now check 
the ($v,S$)-ipw($V,r$) properties for such systems.

 \begin{Proposition}\label{ExP1} Let $2\le p<\infty$. Assume that the system $\D_N:=\{\ff_j\}_{j=1}^N$ satisfies (\ref{Riesz}). Then it has ($v,N$)-ipw($V,r$) with $V=R_1^{-1}$ and $r=1/2$.
 \end{Proposition}
 \begin{proof} Take any subset $A\subset [1,N]\cap \N$, $|A|=v$. For $\{c_i\}_{i=1}^N$ define
 $f:= \sum_{i=1}^N c_i\ff_i$. Then we have
 $$
 \sum_{i\in A} |c_i| \le |A|^{1/2} \left( \sum_{i\in A} |c_i|^2\right)^{1/2} \le |A|^{1/2} \left( \sum_{i=1}^N |c_i|^2\right)^{1/2} 
 $$
 by (\ref{Riesz}) we continue
 $$
 \le |A|^{1/2}R_1^{-1} \left\| \sum_{i=1}^N c_i\ff_i\right\|_2 \le  R_1^{-1}|A|^{1/2}\|f\|_p.
 $$
 \end{proof}

 \begin{Proposition}\label{ExP2} Let $1\le p\le 2$. Assume that the system $\D_N:=\{\ff_j\}_{j=1}^N$ is 
 a uniformly bounded $\|\vi_j\|_\infty \le B$, $j=1,\dots,N$, orthonormal system. Then it has ($v,N$)-ipw($V,r$) with $V=C(B)$ and $r=1/p$.
 \end{Proposition} 
 \begin{proof} As above, take any subset $A\subset [1,N]\cap \N$, $|A|=v$ and for $\{c_i\}_{i=1}^N$ define
 $f:= \sum_{i=1}^N c_i\ff_i$. Consider the following function
 $$
 D_A:= \sum_{i\in A} b_i \ff_i, \qquad b_i :=  c_i/|c_i|,\quad\text{if}\quad c_i\neq 0,\quad \text{and}\quad b_i=0 \quad \text{otherwise}. 
 $$
 Then
 $$
 \sum_{i\in A} |c_i| = \<D_A,f\> \le \|D_A\|_{p'}\|f\|_p,\quad p' := \frac{p}{p-1}.
 $$
 We have that $|b_i| \le 1$ for all $i\in A$. Therefore,
 $$
 \|D_A\|_\infty \le B|A|\quad \text{and} \quad \|D_A\|_2 \le |A|^{1/2}.
 $$
 Using the inequality
 $$
 \|g\|_{p'} \le \|g\|_2^{2/p'} \|g\|_\infty^{1-2/p'},
 $$
 we obtain 
 $$
 \|D_A\|_{p'} \le B^{1-2/p'}|A|^{1/p}
 $$
 and conclude the proof.
 
 \end{proof}
 
 We now proceed to some corollaries of Theorem \ref{CGT2}. We begin with the case $2<p<\infty$.
 The following Theorem \ref{ExT3} is a corollary of Theorem 1.1 of \cite{DTM3}.
 
 \begin{Theorem}[{\cite{DTM3}}]\label{ExT3} Assume that $\D_N$ is a uniformly bounded Riesz system  satisfying (\ref{ub}) and \eqref{Riesz} for some constants $0<R_1\leq R_2<\infty$.
Let $2<p<\infty$, $t\in (0,1]$, and let  $1\leq v\leq N$ be an integer. 
 Then there are constants $V=R_1^{-1}$,   $c=C(t,p)\ge 1$  (depending only on $t$ and $p$, and  $C=C(p,R_1,R_2$) with the following property. 
  	 		
 For any positive integer $v$  satisfying  $u:= \lceil(1+c(2V)^{2}(\ln (2Vv))) v \rceil\leq N$    there exist $m$ points  $\xi^1,\cdots, \xi^m\in  \Og$  with 
$$
		m\leq C       u^{p/2} (\log N)^2
$$  
such that  for any given $f_0\in \cC(\Omega)$,  the WCGA with weakness parameter $t$ applied to $f_0$ with respect to the system  $\D_N(\Omega_m)$ in the space $L_p(\Omega_m,\mu_m)$ provides
\be\label{Ex3}
\|f_{u}\|_{L_p(\Omega_m,\mu_m)} \le C_0\sigma_v(f_0,\D_N(\Omega_m))_{L_p(\Omega_m,\mu_m)}, 
\ee
and
\be\label{Ex4}
\|f_{u}\|_{L_p(\Omega,\mu)} \le C_0\sigma_v(f_0,\D_N)_\infty,
\ee
where $C_0\ge 1$ is an absolute constant.
 \end{Theorem}
\begin{proof} By Proposition \ref{ExP1} the system $\D_N$ has ($v,N$)-ipw($V,r$) with $V=R_1^{-1}$ and $r=1/2$. By Theorem \ref{ExT1} with $\va=1/2$ we obtain that the system $\D_N$ has the $L_p$-usd for the collection $\cX_u(\D_N)$, $u\le N$ with $m$ satisfying (\ref{m}). Let $c$ be the constant from Theorem \ref{CGT2}. For a given $v$ satisfying $(1+c(2V)^{2}(\ln (2Vv))) v\leq N$   we set 
$u:= \lceil (1+c(2V)^{2}(\ln (2Vv)))v \rceil$ and let the set of points $\xi^1,\cdots, \xi^m\in  \Og$ be from Theorem \ref{ExT1}.  

We now apply Theorem \ref{CGT2}. In our case $2<p<\infty$ we have $p^*=2$ and $q^*=2$. 
Therefore, (\ref{mp}) implies (\ref{Ex3}) and (\ref{mp2}) implies (\ref{Ex4}).

\end{proof}
 
 We now formulate a corollary of Theorem \ref{CGT3} (see Corollary 2.1 of \cite{DTM3}). 
 
 \begin{Theorem}[{\cite{DTM3}}]\label{ExT4}
	Under the conditions of Theorem \ref{ExT3}, we have 
	\be\label{Ex5}
	\|f_{u} \|_{L_p(\Omega,\mu)} \le C' \sigma_v(f_0,\CD_N)_{L_p(\Og, \mu_\xi)},
	\ee
	where   $u$ is    from Theorem \ref{ExT3}, $C'$ is a positive absolute constant, and
	$$
	\mu_\xi := \f {\mu+\mu_m}2=\frac{1}{2} \mu + \frac{1}{2m} \sum_{j=1}^m \delta_{\xi^j}.
	$$
\end{Theorem}
 
 We now proceed to the case $1<p\le 2$. 
 
  \begin{Theorem}\label{ExT5}   Assume that the system $\D_N:=\{\ff_j\}_{j=1}^N$ is 
 a uniformly bounded $\|\vi_j\|_\infty \le B$, $j=1,\dots,N$, orthonormal system.
Let $1< p\le 2$, $t\in (0,1]$, and let  $1\leq v\leq N$ be an integer. 
 Then there are constants $V=C(B)$,   $c=C(t,p)\ge 1$, depending only on $t$ and $p$, and  $C=C(p,B)$
 satisfying the following property. 
 
 Set $p':= p/(p-1)$.	 		
 For any positive integer $v$  with  
 $$
 u:= \lceil(1+c(2V)^{p'}(\ln (2Vv))) v^{1/(p-1)}\rceil\leq N
$$
 there exist $m$ points  $\xi^1,\cdots, \xi^m\in  \Og$  with 
\be\label{Ex*}
	 m \le  C u \log N\cdot (\log(2u ))^2\cdot (\log (2u )+\log\log N),
 \ee  
such that for any given $f_0\in \cC(\Omega)$,  the WCGA with weakness parameter $t$ applied to $f_0$ with respect to the system  $\D_N(\Omega_m)$ in the space $L_p(\Omega_m,\mu_m)$ provides
\be\label{Ex6}
\|f_{u}\|_{L_p(\Omega_m,\mu_m)} \le C_0\sigma_v(f_0,\D_N(\Omega_m))_{L_p(\Omega_m,\mu_m)}, 
\ee
and
\be\label{Ex7}
\|f_{u}\|_{L_p(\Omega,\mu)} \le C_0\sigma_v(f_0,\D_N)_\infty,
\ee
where $C_0\ge 1$ is an absolute constant.
 \end{Theorem}
\begin{proof} By Proposition \ref{ExP2} the system $\D_N$ has ($v,N$)-ipw($V,r$) with $V=C(B)$ and $r=1/p$. By Theorem \ref{ExT2} and Remark \ref{ExR1}  we obtain that for large enough $C(p,B)$ the system $\D_N$ has the $L_p$-usd for the collection $\cX_u(\D_N)$, $u\le N$, with $m$ satisfying (\ref{Ex*}). Let $c$ be the constant from Theorem \ref{CGT2}. For a given $v$ satisfying $(1+c(2V)^{p'}(\ln (2Vv))) v^{1/(p-1)}\leq N$   we set 
$u:= \lceil (1+c(2V)^{p'}(\ln (2Vv)))v^{1/(p-1)} \rceil$ and choose the set of points $\xi^1,\cdots, \xi^m\in  \Og$ from Theorem \ref{ExT2} that provides (\ref{Ex2}).  

We now apply Theorem \ref{CGT2}. In our case $1<p\le 2$ we have $p^*=p$ and $q^*=p'$. 
Therefore, (\ref{mp}) implies (\ref{Ex6}) and (\ref{mp2}) implies (\ref{Ex7}).

\end{proof}
 
 We now formulate a corollary of Theorem \ref{CGT3}. 
 
 \begin{Theorem}\label{ExT6}
	Under the conditions of Theorem \ref{ExT5}, we have 
	\be\label{Ex5a}
	\|f_{u} \|_{L_p(\Omega,\mu)} \le C' \sigma_v(f_0,\CD_N)_{L_p(\Og, \mu_\xi)},
	\ee
	where   $u$ is   from Theorem \ref{ExT5}, $C'$ is a positive absolute constant, and
	$$
	\mu_\xi := \f {\mu+\mu_m}2=\frac{1}{2} \mu + \frac{1}{2m} \sum_{j=1}^m \delta_{\xi^j}.
	$$
\end{Theorem}

 {\bf Comments.} The WCGA, which we use in Theorems \ref{ExT3} -- \ref{ExT6}, has two important features: (A) it provides a $u$-term approximation with respect to the given system $\D_N$; (B) it only uses 
 information on function values at a given set of $m$ points. The performance of the algorithm in the sense 
 of (A) is controlled by the corresponding Lebesgue-type inequality, which we will discuss momentarily. 
 The performance of the algorithm in the sense of (B) is controlled by the comparison of its accuracy on a given function class with the optimal sampling recovery error, which we will discuss later in Section \ref{R}. 
We now introduce some concepts from the theory of Lebesgue-type inequalities related to our setting (see, for instance, \cite{VTbookMA}, Section 8.7). 
 
  In a general setting,  we consider an algorithm (i.e., an approximation method) $\A: = \{A_v(\cdot,\D)\}_{v=1}^\infty$ with respect to a given system $\D\subset X$, which is  a  sequence 
of mappings $A_v(\cdot,\D): X\to \Sigma_v(\CD)$, $v=1,2,\cdots$.  Clearly,  
$$
\|f-A_v(f,\D)\|_X \ge \sigma_v(f,\D)_X,\   \   \ v=1,2,\cdots,\   \ f\in X.
$$
We are interested in those pairs $(\D,\A)$ for which the algorithm $\A$ provides approximation  that is close to the best $v$-term approximation.
To be more precise, let  $Y$ be a Banach space such that $Y\subset X$ and $\|\cdot\|_X \le \|\cdot\|_Y$.
The following  definition can be found in \cite[Section 8.7]{VTbookMA} in  the case of $X=Y$ and in \cite{DTM3} in the general case.

\begin{Definition}\label{ExD1}  Let   $\mathbf{a}=\{a(j)\} _{j=1}^\infty$ be a given  sequence of positive integers. 
	We say that a system $\D\subset Y$ is an $\mathbf{a}$-greedy system of depth $u\in\NN$ with respect to an algorithm  $\A = \{A_v(\cdot,\D)\}_{v=1}^\infty$  for the pair $(X,Y)$ if 
\begin{equation}\label{aG2}
\|f-A_{a(v)v}(f,\D)\|_X \le C\sigma_v(f,\D)_Y,\quad v=1,\dots,u,\  \ \forall f\in Y
\end{equation}
for some constant $C> 0$. 
\end{Definition}

Inequalities (\ref{aG2}) are called the Lebesgue-type inequalities for the algorithm $\A$. 

For illustration purposes let us discuss the case, when the system $\D_N:=\{\ff_j\}_{j=1}^N$ is 
 a uniformly bounded $\|\vi_j\|_\infty \le B$, $j=1,\dots,N$, orthonormal system. In this case both 
 Theorem \ref{ExT3} ($2<p<\infty$) and Theorem \ref{ExT5} ($1<p\le 2$) hold. 
 
 Let us begin with the case 
 $2<p<\infty$. By Proposition \ref{ExP1} the system $\D_N$ has the $(v,N)$-ipw$(V,r)$ with $V=1$ and $r=1/2$. Therefore, we can apply Theorem \ref{CGT1} to the WCGA with respect to $\D_N$ in the space 
 $L_p(\Omega,\mu)$. This gives 
 \be\label{Ex6a}
 \|f_{C(t,p)\ln (2v) v}\|_p \le C\sigma_v(f_0,\D_N)_p.
 \ee
 Inequality (\ref{Ex6a}) means that the WCGA provides a very good Lebesgue-type inequality -- 
 the error in the $L_p$ norm after $C(t,p)\ln (2v) v$ iterations is bounded above (in the sense of order) by the best $v$-term approximation in the same $L_p$ norm. However, applying the WCGA in the space 
 $L_p(\Omega,\mu)$, we are not limited to the use of specific (sampling) information about $f_0$. 
 Let us now apply Theorem \ref{ExT3}. In this case we apply the WCGA with respect to the system $\D_N(\Omega_m)$ in the space $L_p(\Omega_m,\mu_m)$, which means that we only use the function values at $m$ points from the set $\Omega_m$. In particular, Theorem \ref{ExT3} gives
 \be\label{Ex7a}
 \|f_{C(t,p,R_1)(\ln (2v)) v}\|_{p} \le C_0\sigma_v(f_0,\D_N)_\infty.
\ee
 Inequality (\ref{Ex7a}) is similar to the (\ref{Ex6a}) with an important difference -- in the right side of (\ref{Ex7a})
 we have the best $v$-term approximation in a stronger norm, namely, in the $L_\infty$ norm. 
 This is the price we pay for limiting ourselves to sampling recovery. Theorem \ref{ExT4} gives a better 
 bound because the $L_p(\Og, \mu_\xi)$ norm is weaker than the $L_\infty$ norm. 
 
 We now proceed to the case 
 $1<p\le 2$. By Proposition \ref{ExP2} the system $\D_N$ has the $(v,N)$-ipw$(V,r)$ with $V=C(B)$ and $r=1/p$. Therefore, we can apply Theorem \ref{CGT1} to the WCGA with respect to $\D_N$ in the space 
 $L_p(\Omega,\mu)$. This gives 
 \be\label{Ex8}
 \left \|f_{C(t,p,B)\ln (2v) v^{\frac{1}{p-1}}}\right\|_p \le C\sigma_v(f_0,\D_N)_p.
 \ee
 Inequality (\ref{Ex8}) means that the WCGA provides a Lebesgue-type inequality -- 
 the error in the $L_p$ norm after $C(t,p,B)\ln (2v) v^{\frac{1}{p-1}}$ iterations is bounded above (in the sense of order) by the best $v$-term approximation in the same $L_p$ norm.  In the case $1<p<2$ we need 
 substantially more iterations than $v$ to obtain (\ref{Ex8}). 
 Let us now apply Theorem \ref{ExT5}. In this case we apply the WCGA with respect to the system $\D_N(\Omega_m)$ in the space $L_p(\Omega_m,\mu_m)$, which means that we only use the function values at $m$ points from the set $\Omega_m$. In particular, Theorem \ref{ExT5} gives
 \be\label{Ex9}
 \left\|f_{c(t,p,B)(\ln (2v)) v^{\frac{1}{p-1}}}\right\|_{p} \le C_0\sigma_v(f_0,\D_N)_\infty.
\ee
 Inequality (\ref{Ex9}) is similar to the (\ref{Ex8}) with an important difference -- in the right side of (\ref{Ex9})
 we have the best $v$-term approximation in a stronger norm, namely, in the $L_\infty$ norm. 
 As above, it is the price we pay for limiting ourselves to sampling recovery. Theorem \ref{ExT6} gives a better bound because the $L_p(\Og, \mu_\xi)$ norm is weaker than the $L_\infty$ norm. 
 
Note that in the case $p=2$ the inequality (\ref{Ex9}) can be improved. It is proved in \cite{DTM2} that 
 \be\label{Ex10}
 \|f_{c(t)  v}\|_{2} \le C_0\sigma_v(f_0,\D_N)_\infty.
\ee
 This means that in the case of sampling recovery with the error in the $L_p$, $1<p<2$, we can use 
 the WCGA with respect to $\D_N$ in the space $L_2(\Omega_m,\mu_m)$. Then $\|f_{c(t)  v}\|_{p} \le \|f_{c(t)  v}\|_{2} $ and (\ref{Ex10}) gives a better bound than (\ref{Ex9}). However, this trick does not work for the 
 Lebesgue-type inequality with the  $\sigma_v(f_0,\CD_N)_{L_p(\Og, \mu_\xi)}$ in the right hand side.

 \section{Optimal recovery on function classes}
 \label{R}
 
 We now discuss an application of the above results to the optimal sampling recovery. We estimate the characteristics $\varrho_m^o(\bF,L_p)$, defined in the Introduction, for a new kind of function classes. 
 We consider  classes of multivariate functions  $\bA^r_\bt(\Psi)$, which were defined in the Introduction. 
 For the reader's convenience we repeat that definition here. 
  For a given $1\le p\le \infty$,
let $\Psi =\{\psi_{\bk}\}_{\bk \in \Z^d}$, $\psi_\bk \in \cC(\Omega)$, $\|\psi_\bk\|_p \le B$, $\bk\in\Z^d$,  be a system in the space $L_p(\Omega,\mu)$. We consider functions representable in the form of an  absolutely convergent series
\be\label{repr}
f = \sum_{\bk\in\Z^d} a_\bk(f)\psi_\bk,\quad \sum_{\bk\in\Z^d} |a_\bk(f)|<\infty.
\ee
Note, that we do not require a unique representation. 
For $\bt \in (0,1]$ define
$$
\left|\sum_{\bk\in\Z^d} a_\bk(f)\psi_\bk\right|_{A_\bt(\Psi)} := \left(\sum_{\bk\in\Z^d} |a_\bk(f)|^\bt\right)^{1/\bt}.
$$
Let $r>0$. We recall the definition of the class $\bA^r_\bt(\Psi)$, which is a class of functions $f$ that have representations (\ref{repr}) satisfying the following conditions
\be\label{Ar}
  \left(\sum_{[2^{j-1}]\le \|\bk\|_\infty <2^j} |a_\bk(f)|^\bt\right)^{1/\bt} \le 2^{-rj},\quad j=0,1,\dots  .
\ee
We need some properties of these classes. 
For an $f$ having representation (\ref{repr}) denote
$$
S_n(f,\Psi) := \sum_{\|\bk\|_\infty< n} a_\bk(f)\psi_\bk.
$$

\begin{Lemma}\label{RL1} Assume that $\|\psi_\bk\|_p \le B$, $\bk\in\Z^d$. Then for $f \in \bA^r_\bt(\Psi)$ there exists a representation (\ref{repr}) such that  
$$
\|f-S_n(f,\Psi)\|_p \le C(r)B 2^{-rn}.
$$
\end{Lemma}
\begin{proof} Let $f\in \bA^r_\bt(\Psi)$. Then by the definition of this class there is a representation (\ref{repr}) 
satisfying (\ref{Ar}). We use that representation. Note that for any $\bt\in (0,1]$ we have $|g|_{A_1(\Psi)}\le |g|_{A_\bt(\Psi)}$. Therefore,
$$
\|f-S_n(f,\Psi)\|_p \le \sum_{j=n}^\infty \|S_{j+1}(f,\Psi)-S_j(f,\Psi)\|_p 
$$
$$
\le  \sum_{j=n}^\infty B|S_{j+1}(f,\Psi)-S_j(f,\Psi)|_{A_1(\Psi)}
$$
$$
\le \sum_{j=n}^\infty B|S_{j+1}(f,\Psi)-S_j(f,\Psi)|_{A_\bt(\Psi)} \le  B\sum_{j=n}^\infty 2^{-r(j+1)} \le C(r)B 2^{-rn}.
$$

\end{proof}

 We now proceed to the best $v$-term approximations. 
 
 \begin{Theorem}\label{RT1} Let $1<p<\infty$, $r>0$, $\bt\in (0,1]$. Assume that $\|\psi_\bk\|_{L_p(\Omega,\mu)} \le B$, $\bk\in\Z^d$, with some probability measure $\mu$ on $\Omega$. Then there exist two constants $c=c(r,\bt,p,d)$ and $C=C(r,\bt,p,d)$ such that 
 for any $v\in \N$ there is a $J\in\N$, $2^J \le v^c$ with the property ($p^*:= \min\{p,2\}$)
 $$
 \sigma_v(\bA^r_\bt(\Psi),\Psi_J)_{L_p(\Omega,\mu)} \le CBv^{1/p^* -1/\bt}, \quad \Psi_J := \{\psi_\bk\}_{\|\bk\|_\infty <2^J}.
 $$
 Moreover, this bound is provided by a simple greedy algorithm.
 \end{Theorem} 
 \begin{proof} Clearly it is sufficient to carry out the prove for $v$ of the order $2^{nd}$, $n\in \N$.  It is known (see, for instance, \cite{VTbook}, p.342) that for any dictionary $\D = \{g\}$, $\|g\|_X \le 1$ in a Banach  space $X$ with $\eta(X,w) \le \gamma w^q$, $1<q\le 2$, we have
 \be\label{R1}
 \sigma_v(A_1(\D),\D)_X \le C(q,\gamma)(v+1)^{1/q -1}.
 \ee
 Here
 $$
 A_1(\D) := \left\{f:\, f=\sum_{i=1}^\infty a_ig_i,\quad g_i\in \D,\quad \sum_{i=1}^\infty |a_i| \le 1 \right\}.
 $$
 Moreover, bound (\ref{R1}) is provided by the WCGA.
 Denote 
 \be\label{R1n}
 \Psi_n := \{\psi_\bk, \,\bk\in \Z^d: \, \|\bk\|_\infty <2^n\},\quad v_n := |\Psi_n|,
 \ee
 $$
 \Psi_{(j)} := \{\psi_\bk,  \,\bk\in \Z^d: \,[2^{j-1}]\le\|\bk\|_\infty <2^j\},\quad j=n+1,\dots.
 $$
  For any sequence $\{v_j\}_{j=n}^\infty$ of
 nonnegative integers such that $v_n+\sum_{j=n+1}^\infty 2v_j \le v$ we obtain
 $$
 \sigma_v(f,\Psi)_X \le \sum_{j=n+1}^\infty \sigma_{2v_j}(S_j(f,\Psi)-S_{j-1}(f,\Psi), \Psi_{(j)})_X+\sigma_{v_n}(S_n(f,\Psi),\Psi_n).
 $$
 Clearly, $\sigma_{v_n}(S_n(f,\Psi),\Psi_n)=0$.
  We now estimate 
  $$
  \sigma_{2v_j}(S_j(f,\Psi)-S_{j-1}(f,\Psi), \Psi_{(j)})_X,\quad \text{for} \quad f\in \bA^r_\bt(\Psi),\quad X=L_p. 
  $$
 We make this approximation in two steps. First, we take $v_j$ coefficients $a_\bk(f)$, $[2^{j-1}]\le \|\bk\|_\infty <2^j$, with the largest absolute values. Denote $I(j)$ the corresponding set of indexes $\bk$. Then, using 
 the known results (see, for instance, \cite{DTU}, p.114, Lemma 7.4.1), we conclude that
 \be\label{R2}
 \sum_{\bk \notin I(j): [2^{j-1}]\le \|\bk\|_\infty <2^j} |a_\bk| \le (v_j+1)^{1-1/\bt}2^{-rj}.
 \ee
 Bounds (\ref{R1}) and (\ref{R2}) and known results on $\eta(L_p,w)$ (see (\ref{CG2})) imply 
 \be\label{R3}
 \sigma_{2v_j}(S_j(f,\Psi)-S_{j-1}(f,\Psi), \Psi_{(j)})_{L_p} \le C(p)B(v_j+1)^{1/p^*-1/\bt}2^{-rj}.
 \ee
 Let $\alpha$ be such that $\al(1/\bt -1/p^*) = r/2$.   Define
 $$
 v_j := [2^{nd-\al(j-n)}],\quad j=0,1,\dots.
 $$
 Then
 $$
 v:=v_n+\sum_{j=n+1}^\infty 2v_j \le 2^{d(n+1)}+ 2\sum_{j=n+1}^\infty 2^{nd-\al(j-n)}\le C_1(r,\bt,p,d)2^{nd}.
 $$
 Also, for $f\in A_\bt^r(\Psi)$ we have
 $$
  \sigma_v(f,\Psi)_{L_p} \le \sum_{j=n+1}^\infty C(p)B(v_j+1)^{1/p^*-1/\bt}2^{-rj}
  $$
  $$
   \le C_2(r,\bt,p,d)B2^{nd(1/p^*-1/\bt) -rn}\le C_3(r,\bt,p,d)B v^{1/p^* -1/\bt -r/d}.
  $$
  Finally, it is clear that $v_j =0$ for $j$ such that $\al(j-n) > nd$. Thus, we can set $J:= \lceil nd/\al\rceil +n+1$. We point out that $J$ does not depend on measure $\mu$. 
  This completes the proof of Theorem \ref{RT1}. 
  
 \end{proof}
 
 We now proceed to corollaries of Theorem \ref{RT1} and Theorems \ref{ExT4} and \ref{ExT6}. We formulate separately a corollary in the case $2<p<\infty$ and a corollary in the case $1<p\le 2$. 
 
 \begin{Theorem}\label{RT2}  Assume that $\Psi$ is a uniformly bounded Riesz system (in the space  $L_2(\Omega,\mu)$) satisfying (\ref{ub}) and \eqref{Riesz} for some constants $0<R_1\leq R_2<\infty$.
Let $2<p<\infty$, $\bt \in (0,1]$, and $r>0$. 
 There exist constants $c'=c'(r,\bt,p,R_1,R_2,d)$ and $C'=C'(r,\bt,p,d)$ such that       for any $v\in\N$ we have the bound   \begin{equation}\label{R4}
 \varrho_{m}^{o}(\bA^r_\bt(\Psi),L_p(\Omega,\mu)) \le C'  v^{1/2 -1/\bt -r/d}   
\end{equation}
for any $m$ satisfying 
$$
m\ge c'    v^{p/2} (\log v)^{2+p/2}.
$$
Moreover, bound (\ref{R4}) is provided by a simple greedy algorithm.
\end{Theorem}
\begin{proof} First, we use Theorem \ref{RT1} in the space $L_p(\Omega,\mu)$. In our case $B=1$ and $2<p<\infty$, which implies that 
$p^*=2$. Consider the system $\Psi_J := \{\psi_\bk\}_{\|\bk\|_\infty <2^J}$, $2^J \le v^c$, from Theorem 
\ref{RT1}. Note, that $J$ does not depend on $\mu$. 
 By Theorem \ref{ExT4} with $\D_N = \Psi_J$  there exist $m$ points  $\xi^1,\cdots, \xi^m\in  \Og$  with 
\be\label{mbound}
		m\leq C       u^{p/2} (\log N)^2,\quad u:= \lceil(1+c(2V)^{2}(\ln (2Vv))) v \rceil,\quad V=R_1^{-1},
\ee  
such that  for any given $f_0\in \cC(\Omega)$,  the WCGA with weakness parameter $t$ applied to $f_0$ with respect to the system  $\D_N(\Omega_m)$ in the space $L_p(\Omega_m,\mu_m)$ provides 
\be\label{R5}
	\|f_{u} \|_{L_p(\Omega,\mu)} \le C' \sigma_v(f_0,\CD_N)_{L_p(\Og, \mu_\xi)}.
	\ee
In order to bound the right side of (\ref{R5}) we apply Theorem \ref{RT1} in the space $L_p(\Og, \mu_\xi)$. 
For that it is sufficient to check that $\|\psi_\bk\|_{L_p(\Og, \mu_\xi)} \le 1$, $\bk\in\Z^d$. This follows from the assumption that $\Psi$ satisfies (\ref{ub}). Thus, by Theorem \ref{RT1} we obtain for $f_0 \in \bA^r_\bt(\Psi)$
\be\label{R6}
 \sigma_v(f_0,\Psi_J)_{L_p(\Og, \mu_\xi)} \le Cv^{1/2 -1/\bt-r/d}.
 \ee
 Combining (\ref{R6}), (\ref{R5}), and taking into account (\ref{mbound}), we complete the proof.

\end{proof}
 
 In the same way we derive from Theorems \ref{RT1} and \ref{ExT6} the following result for $p=2$ and its corollary for $1\le p\le 2$. 
 
  \begin{Theorem}\label{RT3}  Assume that $\Psi$ is a uniformly bounded $\|\psi_\bk\|_\infty \le B$, $\bk\in\Z^d$, orthonormal system with respect to the probability measure $\mu$.
Let $1\le p\le 2$, $\bt \in (0,1]$, and $r>0$. 
 There exist constants $c'=c'(r,\bt,B,d)$ and $C'=C'(r,\bt,d)$ such that       for any $v\in\N$ we have the bound   
 \begin{equation}\label{R7}
 \varrho_{m}^{o}(\bA^r_\bt(\Psi),L_p(\Omega,\mu))\le  \varrho_{m}^{o}(\bA^r_\bt(\Psi),L_2(\Omega,\mu)) \le C' B v^{1/2 -1/\bt-r/d}   
\end{equation}
for any $m$ satisfying 
$$
m\ge c' v (\log(2v))^5.
$$
\end{Theorem}
 {\bf Lower bounds.} Recall the setting 
 of the optimal linear recovery. For a fixed $m$ and a set of points  $\xi:=\{\xi^j\}_{j=1}^m\subset \Omega$, let $\Phi $ be a linear operator from $\bbC^m$ into $L_p(\Omega,\mu)$.
Denote for a class $\bF$ (usually, centrally symmetric and compact subset of $L_p(\Omega,\mu)$) (see \cite{VT51})
$$
\varrho_m(\bF,L_p) := \inf_{\text{linear}\, \Phi; \,\xi} \sup_{f\in \bF} \|f-\Phi(f(\xi^1),\dots,f(\xi^m))\|_p.
$$
In other words, for  given  $m$, a system of functions $\phi_1(\bx),
\dots,\phi_m (\bx)$ from $L_p(\Omega,\mu)$, and a set  
$\xi := \{\xi^1,\cdots,\xi^m\}$ from $\Omega$ we define the linear operator
\be\label{RI1}
 \Phi_m(f,\xi) :=\sum_{j=1}^{m}f(\xi^j)\phi_j (\bx).
\ee
For a class $\bF$ of functions we define the quantity
$$
\Phi_m (\bF,\xi)_p:=\sup_{f\in \bF}\bigl\|\Phi_m (f,\xi) - f\bigr\|_p.
$$
Then
$$
\varrho_m (\bF, L_p) :=\inf_{\phi_1,\dots,\phi_m;\xi^1,\dots,\xi^m}
\Phi_m (\bF,\xi).
$$

 In the above described linear recovery procedure the approximant  comes from a linear subspace of dimension at most $m$. 
It is natural to compare $\varrho_m(\bF,L_p)$  with the 
Kolmogorov widths. Let $X$ be a Banach space and $\bF\subset X$ be a  compact subset of $X$. The quantities  
$$
d_n (\bF, X) :=  \inf_{\{u_i\}_{i=1}^n\subset X}
\sup_{f\in \bF}
\inf_{c_i} \left \| f - \sum_{i=1}^{n}
c_i u_i \right\|_X, \quad n = 1, 2, \dots,
$$
are called the {\it Kolmogorov widths} of $\bF$ in $X$. It is clear that in the case $X=H$ is a Hilbert space the best approximant from a given subspace is provided by the corresponding orthogonal projection. 

We have the following obvious inequality
\be\label{RI2}
d_m (\bF, L_p)\le  \varrho_m(\bF,L_p).
\ee

We now prove the following result.

\begin{Theorem}\label{RT5} Assume that $\Psi$ is  an orthonormal system in $L_2(\Omega,\mu)$. Then we have for any $\bt\in (0,1]$ and $r>0$
$$
d_m(\bA^r_\bt(\Psi),L_2) \asymp m^{-r/d}.
$$
\end{Theorem}
\begin{proof} The corresponding upper bound follows directly from Lemma \ref{RL1}. The lower bound is 
based on the following well known fact. For completeness we give a simple proof of it here. 

\begin{Lemma}\label{RL3} Let $Y=\{y_j\}_{j=1}^N$ be an orthonormal system in a Hilbert space $H$. 
Then we have for $n<N$
$$
d_n(Y,H) \ge (1-n/N)^{1/2}.
$$
\end{Lemma} 
\begin{proof} Take any $n$-dimensional subspace $H_n$ with $n<N$ and assume that $\{h_i\}_{i=1}^n$ is 
an orthonormal basis for $H_n$. Then for each $y_j \in Y$  the best approximation from $H_n$ is given by 
the orthogonal projection of $y_j$ onto $H_n$ and has the error
$$
d(y_j,H_n)_H := \inf_{h\in H_n} \|y_j -h\|_H = \left(\|y_j\|_H^2 - \sum_{i=1}^n |\<y_j,h_i\>|^2\right)^{1/2}.
$$
Therefore,
$$
\sum_{j=1}^N d(y_j,H_n)_H^2 = N - \sum_{j=1}^N\sum_{i=1}^n |\<y_j,h_i\>|^2 
$$
$$
= N - \sum_{i=1}^n\sum_{j=1}^N |\<y_j,h_i\>|^2 \ge N-n.
$$
\end{proof}

For a given $m\in \N$ let $s\in \N$ be such that $|\Psi_{s-1}| <2m \le |\Psi_s| $, where $\Psi_s$ is defined in (\ref{R1n}). Take $Y=\Psi_s$.
 The facts that $2^{-rs}\Psi_s \subset \bA^r_\bt(\Psi)$, $m\asymp 2^{sd}$ and Lemma \ref{RL3} imply 
 Theorem \ref{RT5}. 

\end{proof}

Theorem \ref{RT5} and inequality (\ref{RI2}) imply that we have for any $\bt\in (0,1]$ and $r>0$
$$
\varrho_m(\bA^r_\bt(\Psi),L_2) \gg m^{-r/d}.
$$

Let us discuss lower bounds for the nonlinear characteristic $\varrho_m^o(\bA^r_\bt(\Psi),L_p)$.
We will do it in the special case, when $\Psi$ is the trigonometric system $\Tr^d := \{e^{i(\bk,\bx)}\}_{\bk\in \Z^d}$. Denote for $\bN = (N_1,\dots,N_d)$, $N_j\in \N_0$, $j=1,\dots,d$,
$$
\Tr(\bN,d) := \left\{f=\sum_{\bk: |k_j|\le N_j, j=1,\dots,d} c_j e^{i(\bk,\bx)}\right\},\quad \vartheta(\bN) :=\prod_{j=1}^d (2N_j+1). 
$$

The following Theorem \ref{RT6} is known (see, for instance, \cite{VTbookMA}, p.85).

\begin{Theorem}\label{RT6} Let $\ep \in (0,1]$ and let $T$ be a subspace of $\Tr(\bN,d)$ such that 
$\dim T \ge \ep \dim(\Tr(\bN,d))=\ep \vartheta(\bN)$. Then there is a polynomial $g \in T$ such that 
$$
\|g\|_\infty=1,\qquad \|g\|_2 \ge C(\ep,d) >0.
$$
\end{Theorem}

We now derive from Theorem \ref{RT6} the following lower bound.

\begin{Theorem}\label{RT7} For $\bt\in (0,1]$ and $r>0$ we have
$$
\varrho_m^o(\bA^r_\bt(\Tr^d),L_1) \gg m^{1/2-1/\bt-r/d}.
$$

\end{Theorem}
\begin{proof} We begin with a lemma, which follows from the proof of Lemma 3.6.5 in \cite{VTbookMA}, p.126. 

\begin{Lemma}\label{RL4} Let $\Tr(\bN,d)_\infty$ denote the unit $L_\infty$-ball of the subspace $\Tr(\bN,d)$. Then we have for $m\le \vartheta(\bN)/2$ that
$$
\varrho_m^o(\Tr(\bN,d)_\infty,L_1) \ge c(d) >0  .
$$
\end{Lemma}
\begin{proof} Let a set $\xi\subset \T^d := [0,2\pi)^d$ of points $\xi^1,\dots,\xi^m$ be given. Consider the subspace
$$
T(\xi) := \{f\in \Tr(\bN,d):\, f(\xi^\nu) =0,\quad \nu=1,\dots,m\}.
$$
Then $\dim T(\xi) \ge  \dim(\Tr(\bN,d))/2$ and we apply Theorem \ref{RT6} with $\ep=1/2$. It gives us 
a polynomial $g_\xi \in T(\xi)$ such that $\|g_\xi\|_\infty =1$ and $\|g_\xi\|_2 \ge C(d)>0$. Using the well known 
inequality
$$
\|f\|_2^2 \le \|f\|_1\|f\|_\infty,
$$
we derive that $\|g_\xi\|_1 \ge C(d)^2$. 

Let $\cM$ be a mapping from $\bbC^m$ to $L_1$. Denote $g_0 := \cM(\mathbf 0)$. Then 
$$ 
\|g_\xi -g_0\|_1 +\|-g_\xi - g_0\|_1 \ge 2\|g_\xi\|_1.
$$
This and the fact that both $g_\xi$ and $-g_\xi$ belong to $\Tr(\bN,d)_\infty$ complete the proof of Lemma \ref{RL4}.

\end{proof}
 We now complete the proof of Theorem \ref{RT7}. We take $n\in\N$ and  set $N:=2^n-1$, $\bN := (N,\dots,N)$.
Then for $f\in \Tr(\bN,d)_\infty$ we have by the H{\"o}lder inequality with parameter $2/\bt$
$$
|f|_{A_\bt} = \left(\sum_{\bk: \|\bk\|_\infty\le N} |\hat f(\bk)|^\bt\right)^{1/\bt} \le (2N+1)^{d(1/\bt-1/2)} \left(\sum_{\bk: \|\bk\|_\infty\le N} |\hat f(\bk)|^2\right)^{1/2}
$$
$$
  =(2N+1)^{d(1/\bt-1/2)}\|f\|_2 \le (2N+1)^{d(1/\bt-1/2)}\|f\|_\infty \le (2N+1)^{d(1/\bt-1/2)}.
$$
This bound and Lemma \ref{RL4} imply Theorem \ref{RT7}.

\end{proof}

{\bf Comments.}  Let $\Psi$ be a uniformly bounded orthonormal system. Theorem \ref{RT3} shows that 
for $1\le p \le 2$ the error of optimal recovery $\varrho_m^o(\bA^r_\bt(\Psi),L_p)$ decays not slower than 
$m^{1/2-1/\bt-r/d} (\log m)^C$ with some $C>0$. Theorem \ref{RT7} shows that in the case $\Psi=\Tr^d$
it cannot decay faster than $m^{1/2-1/\bt-r/d}$.  Thus, for uniformly bounded orthonormal systems we know the order of $\varrho_m^o(\bA^r_\bt(\Psi),L_p)$ (up to a logarithmic in $m$ factor) in the case $1\le p\le 2$.  
It would be interesting to understand the behavior of the quantities $\varrho_m^o(\bA^r_\bt(\Psi),L_p)$  in the case $2<p<\infty$ and  $\Psi$ is an arbitrary uniformly bounded orthonormal system.

\section{A generalization of the RIP}
\label{G}

Very recently, it was noticed in \cite{DTM2} that the universal discretization of the $L_2$ norm is closely connected with the concept of Restricted Isometry Property (RIP), which is very important in
compressed sensing. In this section we use the idea of such connection to define a generalization of the concept of  RIP.  

Let $\|\cdot\|_{\ell^m_p}$ be a standard $\ell_p$ norm in $\bbC^m$:
 $$
 \|\bx\|_{\ell^m_p} := \left(\sum_{i=1}^m |x_i|^p\right)^{1/p},\quad \bx=(x_1,\dots,x_m)\in \bbC^m,\quad 1\le p<\infty.
 $$

 We recall the definition of the Restricted Isometry Property (RIP). We formulate it in a convenient for us form. Let $U=\{\bu^j\}_{j=1}^N$ be a system of column vectors   from
 $\bbC^m$. We form a matrix $\bU:= [\bu^1\  \bu^2\  ...\  \bu^N]$ with vectors $\bu^j$ being the columns of this matrix. We say that matrix $\bU$  or the system $U$ has the RIP  with parameters $v$ and $\delta \in (0,1)$ if for any subset $J$ of indexes from $\{1,2,\dots,N\}$ with cardinality $|J|\le v$ we 
 have for any vector $\ba=(a_1,\dots,a_N)\in \bbC^N$ of coefficients such that $a_j=0, j\notin J$ the following inequalities
 $$
 (1-\delta)\|\ba\|_{\ell_2^N} \le \left\|\sum_{j\in J} a_j\bu^j\right\|_{\ell_2^m} \le  (1+\delta)\|\ba\|_{\ell_2^N}.
 $$
 In such a case we also write $U\in RIP(v,\delta)$ or $\bU\in RIP(v,\delta)$. 
 
 We now give a generalization of this concept.   Let $\|\cdot\|$ be a norm in $\bbC^N$. 
 \begin{Definition}\label{GD1} We say that matrix $\bU$  or the system $U$ has the\newline RIP($\ell_p,\|\cdot\|$)   with parameters $v$ and $\delta \in (0,1)$ if for any subset $J$ of indexes from $\{1,2,\dots,N\}$ with cardinality $|J|\le v$ we 
 have for any vector $\ba=(a_1,\dots,a_N)\in \bbC^N$ of coefficients such that $a_j=0, j\notin J$ the following inequalities
 $$
 (1-\delta)\|\ba\| \le \|\bU \ba\|_{\ell^m_p} \le  (1+\delta)\|\ba\|.
 $$
 In such a case we also write $U\in RIP(\ell_p,\|\cdot\|,v,\delta)$ or $\bU\in RIP(\ell_p,\|\cdot\|,v,\delta)$. 
  \end{Definition}
  In the case $p=2$ and $\|\cdot\| = \|\cdot\|_{\ell^N_2}$ the above definition coincides with the definition of the RIP with parameters $v$ and $\delta$. 
  
 Let $\Omega$ be a compact subset of $\R^d$ with the probability measure $\mu$. 
  For a system $\D_N=\{g_i\}_{i=1}^N$ and a set of points $\xi:= \{\xi^j\}_{j=1}^m \subset \Omega $ consider the system of vectors $G_N(\xi):=\{\bar g_i \}_{i=1}^N$ where 
  $$
 \bar g_i := \bar g_i(\xi) :=m^{-1/p} (g_i(\xi^1),\dots,g_i(\xi^m))^T,\quad  i=1,\dots,N.
  $$
  The reader can find results on the RIP properties of systems $G_N(\xi)$ associated with uniformly bounded orthonormal systems $\D_N$ in the book \cite{FR}.  
 For a vector $\ba\in \bbC^N$ define 
 \be\label{G1}
 \|\ba\| := \|f(\ba,\D_N)\|_{L_p(\Omega,\mu)}, \quad f:= f(\ba,\D_N) := \sum_{i=1}^N a_ig_i.
 \ee
 Then 
  \be\label{G2}
\left\|\sum_{i=1}^N a_i\bar g_i\right\|_{\ell^m_p}^p = \frac{1}{m}\sum_{j=1}^m |f(\xi^j)|^p.
\ee
The RIP($\ell_p,\|\cdot\|$)   with parameters $v$ and $\delta \in (0,1)$ for the system $U:=G_N(\xi)$ reads as follows
 \be\label{G3}
(1-\delta) \|\ba\| \le \|\bU\ba\|_{\ell^m_p} \le  (1+\delta)\|\ba\|.
\ee
Using (\ref{G1}) and  (\ref{G2}) we rewrite  (\ref{G3}) in the following way
 \be\label{G4}
(1-\delta)\|f\|_{L_p(\Omega,\mu)} \le \left(\frac{1}{m}\sum_{j=1}^m |f(\xi^j)|^p\right)^{1/p}  \le  (1+\delta)\|f\|_{L_p(\Omega,\mu)}.
\ee

We now formulate the problem of universal discretization (see \cite{VT160}). 

{\bf The  problem of universal discretization, $1\le p <\infty$.} Let $\cX:= \{X(n)\}_{n=1}^k$ be a collection of finite-dimensional  linear subspaces $X(n)$ of the $L_p(\Omega,\mu)$, $1\le p < \infty$. We say that a set $\xi:= \{\xi^j\}_{j=1}^m \subset \Omega $ provides {\it universal discretization} for the collection $\cX$ if there are two positive constants $C_i$, $i=1,2$, such that for each $n\in\{1,\dots,k\}$ and any $f\in X(n)$ we have
$$
C_1\|f\|_p^p \le \frac{1}{m} \sum_{j=1}^m |f(\xi^j)|^p \le C_2\|f\|_p^p.
$$
 
It is clear that (\ref{G4}) is exactly the statement on the universal discretization provided by the set $\xi:= \{\xi^j\}_{j=1}^m \subset \Omega $ for the collection $\cX_v(\D_N)$ in the case $C_1 = (1-\delta)^p$ and $C_2 = (1+\delta)^p$. 
 
 We now formulate a corollary of known results on universal discretization for the RIP$(\ell_p,\|\cdot\|,v,\delta)$. The following theorem is a corollary of Theorem \ref{ExT2}. 
 
   \begin{Theorem}\label{GT1} Let $1\le p\le 2$. Assume that $ \D_N=\{\ff_j\}_{j=1}^N\subset L_\infty(\Og)$ is a  system  satisfying  the conditions  \eqref{ub} and   \eqref{Bessel} for some constant $K\ge 1$. Let $\xi^1,\cdots, \xi^m$ be independent 
 	random points on $\Og$  that are  identically distributed  according to  $\mu$. 
 	 Then for any $\delta \in (0,1)$ there exist constants  $C=C(p,\delta)>1$ and $c=c(p,\delta)>0$ such that given any   integers  $1\leq v\leq N$ and 
 	 $$
 	 m \ge  C Kv \log N\cdot (\log(2Kv ))^2\cdot (\log (2Kv )+\log\log N),
 	 $$
the system $\varPhi_N(\xi)=\{\bar \ff_i(\xi) \}_{i=1}^N$, where 
  $$
  \bar \ff_i(\xi) :=m^{-1/p} (\ff_i(\xi^1),\dots,\ff_i(\xi^m))^T,\quad  i=1,\dots,N,
  $$ 
has the RIP($\ell_p,\|\cdot\|$)   with parameters $v$ and $\delta$	 
  with probability $$\ge 1-2 \exp\Bl( -\f {cm}{Kv\log^2 (2Kv)}\Br).$$
\end{Theorem}
  
\section{Discussion}
\label{D}

We begin with a comment on greedy approximation. The known Theorem \ref{CGT1} plays an important 
role in our study. This theorem guarantees certain error bound on the residual of the WCGA applied with respect to a system satisfying Definition \ref{ID2}. We now cite another result on greedy approximation, which gives somewhat better bound than in Theorem \ref{CGT1} under an extra condition on the system.
We formulate the corresponding results.

\begin{Definition}\label{DD1}  Let $X$ be a real Banach space with a norm $\|\cdot\|$. We say that a system $\D=\{g_i\}_{i=1}^\infty\subset X$ is   ($v,S$)-unconditional with parameter $U$ in $X$  if for any $A\subset B$ with  $|A|\le v$ and  $|B|\le S$,  and  for any $\{c_i\}_{i\in B}\subset \CC$, we have
\be\label{unc}
\left\|\sum_{i\in A} c_ig_i\right\| \le U\left\|\sum_{i\in B} c_ig_i\right\|.
\ee
\end{Definition}
We gave Definition \ref{DD1} for a countable system $\D$. Similar definition can be given for any system $\D$ as well. 

The following Theorem \ref{DT1} is an analog of Theorem \ref{CGT1}.

 \begin{Theorem}[{\cite[Theorem 2.8]{VT144}, \cite[Theorem 8.7.18]{VTbookMA}}]\label{DT1} Let $X$ be a real Banach space satisfying that  $\eta(X, w)\le \gamma w^q$, $w>0$ for some parameter $1<q\le 2$. Suppose that $\D\subset X$  is a system in $X$ with the  ($v,S$)-incoherence property for some integers $1\leq v\leq S$ and  parameters   $V>0$ and $r>0$. Also, assume that the system $\D$ is ($v,S$)-unconditional with parameter $U$.
 Then the WCGA with weakness parameter $t$ applied to $f_0$ with respect to the system $\CD$ provides
$$
\|f_{u}\| \le C\sigma_v(f_0,\D)_X,\quad u:= \lceil C(t,\gamma,q)V^{q'}\ln (U+1) v^{rq'}\rceil, 
$$
for any positive integer  $v$ satisfying 
$
v+u\le S,
$ 
where 
$$
q':=\f q{q-1},\   \ C(t,\gamma,q) = C(q)\gamma^{\frac{1}{q-1}}  t^{-q'},
$$ 
and $C>1$ is an absolute constant. 
\end{Theorem}

 In the same way as Proposition \ref{CGP1} was proved in Section \ref{CG} one can prove the following 
 statement.
 
 \begin{Proposition}\label{DP1} Let $1\le p<\infty$. Assume that the system $\D_N = \{g_i\}_{i=1}^N$ is ($v,S$)-unconditional with parameter $U$ in $L_p(\Omega,\mu)$ and that it has the $L_p$-usd for the collection $\cX_u(\D_N)$, $v\le u\le S\le N$. Then the system $\D_N(\Omega_m)$, which is the restriction of the $\D_N$ onto $\Omega_m$, is ($v,u$)-unconditional with parameter $U3^{1/p}$ in $L_p(\Omega_m,\mu_m)$ and 
\be\label{D2}
\|g_i\|_{L_p(\Omega_m,\mu_m)} \le (3/2)^{1/p} \|g_i\|_p.
\ee
\end{Proposition}
 
 Thus, we obtain an analog of Theorem \ref{CGT2} with an additional assumption of the ($v,S$)-unconditionality and with $\ln(Vv)$ replaced by $\ln(U+1)$. We formulate it here for completeness. 
 Note, that we consider real Banach spaces here.
 
 \begin{Theorem}\label{DT2} Let $1<p<\infty$ and $p^* := \min(p,2)$, $q^* := p^*/(p^*-1)$. Assume that the system $\D_N = \{g_i\}_{i=1}^N$ has the ($v,S$)-ipw($V,r$) and the $L_p$-usd for the collection $\cX_u(\D_N)$, $v\le u\le S\le N$. Also assume that it is ($v,S$)-unconditional with parameter $U$ in $L_p(\Omega,\mu)$.
 
Suppose that $\xi=\{\xi^1,\cdots, \xi^m\}\subset \Og $  is a set  of $m$ points
in $\Og$ that provides the   $L_p$-usd  for the collection $\cX_u(\D_N)$.

Then there exists a constant  $c=C(t,p)\ge 1$ depending only on $t$ and $p$ such that for any positive integer $v$  with  $v+v'\leq u$, $v':=\lceil c(2V)^{q^*}(\ln (3U+1)) v^{rq^*}\rceil$, and for any given $f_0\in \cC(\Omega)$,  the WCGA with weakness parameter $t$ applied to $f_0$ with respect to the system  $\D_N(\Omega_m)$ in the space $L_p(\Omega_m,\mu_m)$ provides
\be\label{mp'}
\|f_{v'}\|_{L_p(\Omega_m,\mu_m)} \le C\sigma_v(f_0,\D_N(\Omega_m))_{L_p(\Omega_m,\mu_m)}, 
\ee
and
\be\label{mp2'}
\|f_{v'}\|_{L_p(\Omega,\mu)} \le CD\sigma_v(f_0,\D_N)_\infty,
\ee
where $C\ge 1$ is an absolute constant.
 \end{Theorem}
 
 \begin{Theorem}\label{DT3}
	Under the conditions of Theorem \ref{DT2}, we have 
	\be\label{D3}
	\|f_{v'} \|_{L_p(\Omega,\mu)} \le C' \sigma_v(f_0,\CD_N)_{L_p(\Og, \mu_\xi)},
	\ee
	where   $v'$ is    from Theorem \ref{DT2}, $C'$ is a positive absolute constant, and
	$$
	\mu_\xi := \f {\mu+\mu_m}2=\frac{1}{2} \mu + \frac{1}{2m} \sum_{j=1}^m \delta_{\xi^j}.
	$$
\end{Theorem}
 
 In particular, this approach allows us to slightly improve Theorem \ref{RT3}. 
 
   \begin{Theorem}\label{DT4}  Assume that $\Psi$ is a uniformly bounded $\|\psi_\bk\|_\infty \le B$, $\bk\in\Z^d$, orthonormal system.
Let $1< p\le 2$, $\bt \in (0,1]$, and $r>0$. 
 There exist constants $c'=c'(r,\bt,B,d)$ and $C'=C'(r,\bt,d)$ such that       for any $v\in\N$ we have the bound   \begin{equation}\label{D4}
 \varrho_{m}^{o}(\bA^r_\bt(\Psi),L_p(\Omega,\mu)\le  \varrho_{m}^{o}(\bA^r_\bt(\Psi),L_2(\Omega,\mu) \le C' B v^{1/2 -1/\bt-r/d}   
\end{equation}
for any $m$ satisfying 
$$
m\ge c' v (\log(2v))^4.
$$
\end{Theorem}
 
  We now discuss the role of conditions coming from Definitions \ref{ID1}, \ref{ID2}, and \ref{DD1} and 
 the limitations of the techniques based on them. 
 
  {\bf Discretization.} Probably, the universal sampling discretization assumption (Definition \ref{ID1}) is the most restrictive one. Known results show that results on universal 
 sampling discretization of the $L_p$ norm are different in the cases (I) $2<p<\infty$ and (II) $1\le p\le 2$.
 
 Theorem \ref{ExT1} applies in the case (I) and gives a sufficient condition for the universal sampling discretization of the $L_p$ norm 
 for the collection $\cX_u(\D_N)$ in the form $m\ll u^{p/2} (\log N)^2$. The power $p/2$ is greater than $1$ 
 in this case. It is known that we cannot make this power smaller. The following Proposition \ref{DP2} is from \cite{KKLT} 
(see D.20. A Lower bound there). 

\begin{Proposition}\label{DP2} Let $p\in (2,\infty)$ and let a subspace $X_N \subset \cC(\Omega)$
be such that the $L_p(\Omega,\mu)$ is equivalent to the $L_2(\Omega,\mu)$. Then it is necessary 
to have at least $N^{p/2}$ (in the sense of order) points for discretization with positive weights of the 
$L_p(\Omega,\mu)$ norm on $X_N$. 
\end{Proposition}
 
 In particular, one can take as an example in Proposition \ref{DP2} a subspace $X_N$ spanned by 
 $\{e^{ikx}\}_{k\in \Lambda_N}$, where $\Lambda_N = \{k_j\}_{j=1}^N$ is a lacunary sequence: $k_1=1$, $k_{j+1} \ge bk_j$, $b>1$, $j=1,\dots,N-1$. This means that we cannot improve the power $p/2$ even if we 
 discretize a single subspace. We also need to pay some price for universal discretization (see \cite{DTM2}, Section 4). 
 
 Theorem \ref{ExT2} applies in the case (II) and gives a much better (compared to Theorem \ref{ExT1}) sufficient condition for the universal sampling discretization of the $L_p$ norm 
 for the collection $\cX_u(\D_N)$ in the form (roughly) $m\ll u (\log N)^4$. 
 This means that we cannot substantially improve the universal sampling discretization results that we use. 
 What probably can be done in this direction is to widen the class of systems $\D_N$, for which we 
 prove the universal sampling discretization results. The uniform boundedness condition on $\D_N$ is 
 a strong condition. 
 
 {\bf Recovery algorithms.} We now discuss recovery algorithms. We build a recovery algorithm in several steps. First, we take a system $\D_N$ and specify a target accuracy in the right side of the Lebesgue-type 
 inequality, namely, the best $v$-term approximation of $f_0$ in some norm, say, in the uniform norm. 
 So, we get a parameter $v$. Second, we obtain a parameter $u$, which gives a number of iteration of 
 the algorithm needed for the Lebesgue-type inequality. Third, for sampling recovery we need to discretize 
 the domain $\Omega$. At this step we use the universal sampling discretization for the collection $\cX_u(\D_N)$. The known results provide us a bound on $m$ in terms of $u$ and $N$. 
 
 We pointed out above that at the third step we loose almost nothing in the case $1\le p\le2$ ($m\ll u (\log N)^4$) and loose in the sense of power in the case $2<p<\infty$ ($m\ll u^{p/2} (\log N)^2$). 
 The situation is a kind of opposite, when we apply the WCGA. Theorem \ref{ExT3} shows that 
 in the case (I) ($2<p<\infty$) we need $u\asymp  v\ln (2v)$ iterations. Theorem \ref{ExT5} shows that  in the case (II) ($1<p\le 2$) we need $u\asymp v^{1/(p-1)}\ln (2v)$ iterations. 
 
 The WCGA is an efficient algorithm, which is practically feasible. Its version for $p=2$ -- Weak Orthogonal Matching Pursuit -- is widely used in practical applications. We now describe an algorithm from \cite{DTM3}, which is not as practical as the WCGA, but it gives better accuracy bounds. 
 
  For brevity denote $L_p(\xi) := L_p(\Omega_m,\mu_m)$, where $\Omega_m=\{\xi^\nu\}_{\nu=1}^m$  and  $\mu_m(\xi^\nu) =1/m$, $\nu=1,\dots,m$. Let 
$B_v(f,\D_N,L_p(\xi))$ denote the best $v$-term approximation of $f$ in the $L_p(\xi)$ norm with 
respect to the system $\D_N$. Note that $B_v(f,\D_N,L_p(\xi))$ may not be unique. Obviously,
\be\label{D5}
\|f-B_v(f,\D_N,L_p(\xi))\|_{L_p(\xi)} = \sigma_v(f,\D_N)_{L_p(\xi)}.
\ee

We proved in \cite{DTM3} the following theorem.

 \begin{Theorem}[{\cite{DTM3}}]\label{DT5} Let $1\le p<\infty$ and let $m$, $v$, $N$ be given natural numbers such that $2v\le N$.  Let $\D_N\subset \C(\Og)$ be  a system of $N$ elements. Assume that  there exists a set $\xi:= \{\xi^j\}_{j=1}^m \subset \Omega $, which provides {\it one-sided $L_p$-universal discretization} 
  \be\label{D6}
 \|f\|_p \le D\left(\frac{1}{m} \sum_{j=1}^m |f(\xi^j)|^p\right)^{1/p}, \quad \forall\, f\in \Sigma_{2v}(\D_N), 
\ee
  for the collection $\cX_{2v}(\D_N)$. Then for   any  function $ f \in \C(\Omega)$ we have
\be\label{D7}
  \|f-B_v(f,\D_N,L_p(\xi))\|_p \le 2^{1/p}(2D +1) \sigma_v(f,\D_N)_{L_p(\Og, \mu_\xi)}
 \ee
 and
 \be\label{D8}
  \|f-B_v(f,\D_N,L_p(\xi))\|_p \le  (2D +1) \sigma_v(f,\D_N)_\infty.
 \ee
 \end{Theorem}

 Theorem \ref{DT5} shows that the recovery algorithm $B_v(\cdot,\D_N,L_p(\xi))$ does not loose anything in the sense of number of iterations and only requires the universal sampling discretization for the collection $\cX_{2v}(\D_N)$. However, the WCGA at each iteration searches over at most $N$ dictionary elements for choosing a new one. On the other hand, the algorithm $B_v(\cdot,\D_N,L_p(\xi))$ performs  $\binom{N}{v}$ iterations of the $\ell_p$ projections on the $v$-dimensional subspaces. 
 
 \begin{Remark}\label{DR1} Theorem \ref{DT5} allows us to slightly improve bounds on $m$ in Theorems \ref{RT2}  and   \ref{RT3}.
 Namely, in Theorem \ref{RT2} instead of $m\ge c'    v^{p/2} (\log v)^{2+p/2}$ we can write $m\ge c'    v^{p/2} (\log v)^{2}$. In Theorem  \ref{RT3}  instead of $m\ge c'    v (\log v)^5$ we can write $m\ge c'    v (\log v)^4$. 
\end{Remark}

  \Addresses

\end{document}